\documentclass[reqno,12pt]{amsart}
\usepackage{latexsym}
\usepackage{amsfonts}
\usepackage{amssymb}
\usepackage{mathrsfs}
\usepackage[all]{xy}
\usepackage{bbm}
\usepackage{color}
\usepackage[mathscr]{euscript}
\usepackage{tikz}
\usepackage{amsmath,amsthm,amssymb,mathrsfs,amsfonts,verbatim,color,leftidx}
\usepackage{amsbsy}
\usepackage[colorlinks,urlcolor=blue,linkcolor=blue,citecolor=brown]{hyperref}
\usetikzlibrary{matrix,arrows}

\newtheorem{theorem}{Theorem}[section]
\newtheorem{corollary}[theorem]{Corollary}
\newtheorem{lemma}[theorem]{Lemma}

\newtheorem{problem}[theorem]{Problem}
\newtheorem{question}[theorem]{Question}
\newtheorem{proposition}[theorem]{Proposition}
\theoremstyle{definition}

\newtheorem{definition}[theorem]{Definition}
\newtheorem{remark}[theorem]{Remark}
\numberwithin{equation}{section}

\newcommand{\bdelta}{\boldsymbol{\delta}}
\newcommand{\bTheta}{\boldsymbol{\Theta}}

\newcommand{\RR}{\mathbb{R}}

\newcommand{\NN}{\mathbb{N}}

\newcommand{\QQ}{\mathbb{Q}}
\newcommand{\ZZ}{\mathbb{Z}}

\newcommand{\cI}{\mathcal{I}}

\newcommand{\cS}{\mathcal{S}}

\newcommand{\cQ}{\mathcal{Q}}

\newcommand{\cA}{\mathcal{A}}

\newcommand{\cM}{\mathcal{M}}

\newcommand{\cH}{\mathcal{H}}

\newcommand{\ba}{\mathbf{a}}
\newcommand{\bb}{\mathbf{b}}

\newcommand{\bd}{\mathbf{d}}

\newcommand{\bx}{\mathbf{x}}
\newcommand{\bh}{\mathbf{h}}

\newcommand{\bv}{\mathbf{v}}

\newcommand{\bM}{\mathbf{M}}

\newcommand{\bp}{\mathbf{p}}

\newcommand{\bL}{\mathbf{L}}

\newcommand{\bq}{\mathbf{q}}

\newcommand{\dist}{\mathrm{dist}}

\newcommand{\SL}{\mathrm{SL}}

\newcommand{\dd}{\; \mathrm{d}}

\newcommand{\bone}{\mathbbm{1}}

\DeclareMathOperator{\e}{e}

\newcommand{\eps}{\varepsilon}

\newcommand{\Sing}{\mathbf{Sing}}

\newcommand {\ignore}[1]  {}

\newif\ifdraft\drafttrue

\draftfalse

\newcommand{\ggm}{G/\Gamma}

\newcommand{\btheta}{{\boldsymbol{\theta}}}

\renewcommand{\emptyset}{\varnothing}
\renewcommand{\setminus}{\smallsetminus}

\newcommand{\red}[1]{\textcolor[rgb]{1.00,0.00,0.00}{#1}}

\begin{document}

\title{Singular points for cone actions on the product of certain homogeneous spaces}

\author{Lifan Guan}
\address{Institute for Theoretical Sciences, School of Science, Westlake University,
600 Dunyu Road, Sandun town, Xihu district, Hangzhou 310030, Zhejiang Province, China.
}
\email{guanlifan@westlake.edu.cn}

\author{Chengyang Wu}
\address{School of Mathematical Sciences, Peking University, Haidian District, Beijing, 100871, China}
\email{chengyangwu1999@gmail.com}

\thanks{}

\begin{abstract}
In this paper, we investigate divergent orbits for \emph{cone} actions on products of certain homogeneous spaces. We introduce a notion of \emph{essential singularity} for such actions, and estimate the Hausdorff dimension of the corresponding singular set. In particular, let
$G/\Gamma=\SL(2,\RR)^s/\SL(2,\ZZ)^s,$
and let $C$ be a cone in the positive Weyl chamber with angular aperture $\epsilon>0$. Then the Hausdorff dimension of the set of points with essential divergent orbits under $C$ satisfies that when $\epsilon\in (0,\frac{1}{64})$,
$$
3s-\frac{1}{2}-4(s-1)\epsilon \le \dim D^e(C, G/\Gamma)\le 3s-\frac{1}{2}-\frac{1}{3}\epsilon.
$$
This extends the previous result of An--Guan--Marnat--Shi \cite{AGMS} to higher-dimensional cone actions.
\end{abstract}

\maketitle

\section{Introduction}

\subsection{Dimension of singular matrices}
One of the central objectives in the field of Diophantine approximation is to study how well a real matrix can be approximated by rational ones. Let $m,n$ be natural numbers, and let $M_{m\times n}(\RR)$ denote the space of $m\times n$ matrices with real entries. The classical Dirichlet's theorem states that for any matrix $\btheta\in M_{m\times n}(\RR)$ and any $Q>1$, the system of  inequalities
\begin{equation}\label{Dir}\tag{D}
	\left \{
	\begin{array}{l}
		\| \btheta \bq - \bp \| < Q^{-n/m}\\
		0< \|\bq\| \le Q
	\end{array}  \right.
\end{equation}
admits a solution $(\bp, \bq) \in \ZZ^m\times  \ZZ^n$. Here and hereafter, $\|\cdot\|$ stands for the supremum norm on an Euclidean space.

In this paper we focus on \textsl{singular matrices}, namely those matrices for which Dirichlet's theorem can be ``infinitely improved''. To be precise, a matrix $\btheta \in M_{m\times n}(\RR)$ is called \emph{singular} if for any $\epsilon>0$ and any sufficiently large $Q>1$, the system of inequalities
\begin{equation}\label{Sing}\tag{S}
\left \{
\begin{array}{l}
\| \btheta \bq - \bp \| < \epsilon Q^{-n/m}\\
0< \|\bq\| \le Q
\end{array}  \right.
\end{equation}
admits a solution $(\bp, \bq) \in \ZZ^m\times  \ZZ^n$.  Let $\Sing_{m,n}$ denote the set of all singular matrices in $M_{m\times n}(\RR)$. It is well-known that $\Sing_{1,1}=\QQ$.


For general $(m,n)\neq (1,1)$, a natural problem is to quantify the size of $\Sing_{m,n}$. In 1948, Khintchine \cite{KhinSing2} proved that the Lebesgue measure of $\Sing_{m,n}$ is $0$. The finer problem of determining the Hausdorff dimension of $\Sing_{m,n}$ was fully resolved only much more recently, within the last fifteen years \cite{Ch,ChCh,KKLM,VarPrinc}.

 \begin{theorem}\label{T:dfsu}
	For any $(m,n)\in \NN^2$ with $(m, n)\ne (1, 1)$, \begin{equation*}\label{E:dimmn}
		\dim \Sing_{m,n} = mn-\frac{mn}{m+n}.
	\end{equation*}
\end{theorem}

We emphasize that the sharp upper bound for $\dim \Sing_{m,n}$ is due to Kadyrov, Kleinbock, Lindenstrauss, and Margulis \cite{KKLM}, whose arguments rely on the contraction property of a height function; whereas the sharp lower bound was obtained by Das, Fishman, Simmons, and Urba\'nski \cite{VarPrinc} through developing a variational principle in parametric geometry of numbers.

A widely open question is to extend Theorem \ref{T:dfsu} to the set of \textsl{weighted singular matrices}. One may refer to \cite{LSST,KP} for the case $n=1$, especially the lower bound of its Hausdorff dimension. A very recent paper \cite{AG26} deals with the general upper bound of its packing dimension.

\subsection{Dynamical interpretations of singularity}
In view of Dani's correspondence \cite{Dani}, many Diophantine properties of real matrices can be recast dynamically. Let $G=\SL(m+n, \RR)$, $\Gamma=\SL(m+n, \ZZ)$, $Y_{m+n}=\ggm$, and
\begin{equation}\label{eq;f+}
  F_{m,n}^+=\{g_t^{(m,n)}: t\ge 0\}, \quad\text{ where }\quad g_t^{(m,n)}=\left(\begin{array}{cc}\e^{t/m}I_m &   \\   & \e^{-t/n}I_n\end{array}\right)\in G.\footnote{Here the time parametrization of $g_t^{(m,n)}$ is chosen to be compatible with the definition of a template in \cite{VarPrinc}. It differs from that in \cite{KKLM} by a factor $mn$.}
\end{equation}
For $\btheta\in M_{m\times n}(\RR)$, set
\begin{equation}\label{eq;utheta}
u_{\btheta}=\begin{pmatrix}
                I_m & \btheta \\
                0 & I_n
              \end{pmatrix}\in G \quad\text{ and }\quad x_{\btheta}=u_{\btheta}\ZZ^{m+n}\in Y_{m+n}.
 \end{equation}
Then $\btheta\in M_{m\times n}(\RR)$ is singular if and only if the trajectory $F_{m,n}^+ x_{\btheta}$ is \emph{divergent}, i.e. it eventually leaves every compact subset of $Y_{m+n}$.

In general, let $G$ be a noncompact Lie group, $\Gamma$ a nonuniform lattice in $G$, and $F^+=\{g_t: t\ge 0\}$ a one-parameter subsemigroup in $G$.
We say that a point $x\in \ggm$ is \emph{$F^+$-singular} if the corresponding trajectory $F^+x$ is divergent on $\ggm$.
The set $D(F^+, \ggm)$ of {$F^+$-singular} points has been extensively investigated in recent years.  

In the most general case, it was proved in \cite{GS} that $\dim D(F^+, \ggm)<\dim \ggm$. A direct application of Theorem \ref{T:dfsu} gives for any $(m,n)\in \NN^2$ with $(m,n)\ne (1,1)$,
\begin{equation*}\label{e-dim-dmn}
 \dim D(F_{m,n}^+, Y_{m+n})=\dim Y_{m+n}-\frac{mn}{m+n}.
\end{equation*}
Furthermore, one may consider the interesting case when $(F^+, \ggm)$ is a product of homogeneous dynamical systems as follows.

More precisely, let $s\geq 2$. For $1\leq i\leq s$, let $(m_i, n_i)\in \NN^2$, and 
\begin{equation*}
	G_i=\SL(m_i+n_i, \RR), \quad\Gamma_i=\SL(m_i+n_i, \ZZ).
\end{equation*}
Write
\begin{equation}\label{E:notation}
G=\prod_{i=1}^s G_i, \quad \Gamma=\prod_{i=1}^s \Gamma_i, \quad X_i=G_i/\Gamma_i, \quad X=\ggm=\prod_{i=1}^s X_i,
\end{equation} 
The dynamics on $X$ is given by a flow
\begin{equation}\label{E:flow}
	F^+_{\ba}:=\left\{g_{\ba t}=\left(g_{a_1t}^{(m_1,n_1)},\ldots, g_{a_st}^{(m_s,n_s)} \right): t\ge 0\right\},
\end{equation}
where $\ba=(a_1,\cdots,a_s)\in \RR^s_+:=(0,+\infty)^s$. The following result was established by An, Guan, Marnat, and Shi in \cite{AGMS}: 

\begin{theorem}\label{T:AGMS}
	Let $(F_\ba^+,X)$ be given as in \eqref{E:notation} and \eqref{E:flow}. Then
	\begin{equation*}
		\dim D(F^+_\ba, X)=\dim X- \min_{1\le i\le s} \frac{m_in_i}{m_i+n_i}.
	\end{equation*}
\end{theorem}

In \cite{AGMS}, the authors naturally introduced a concept of ``essential singularity'' to avoid some degenerate cases for the homogeneous dynamical system $(F_\ba^+,X)$ as in \eqref{E:notation} and \eqref{E:flow}. Here we rearrange their settings and restate their main results. 

Let $s\geq 2$. For any $s'\in \{0,1,\cdots,s-1\}$, consider the collection $$
\cI_{s'}:=\{I\subseteq \{1,\cdots,s\}:\#I=s'\}.
$$
For any $I\in \cI_{s'}$, we write $\pi_I$ to be the natural projection from $G$ to $\prod_{i\in I}G_i$, as well as the natural projection from $X$ to $\prod_{i\in I}X_i$. For the sake of simplicity, we write $\pi_i=\pi_{\{i\}}$ and $\pi_{\widehat{i}}=\pi_{\{1,\cdots,s\}\setminus\{i\}}$ for any $i\in\{1,\cdots,s\}$.

Note that for a point $\bx=(x_1,\cdots,x_s)\in X$, and any $\varnothing\neq I\subseteq I'\subseteq\{1,\cdots,s\}$, \begin{equation*}
	\pi_I(\bx) \text{ is }\pi_I(F_{\ba}^+)\text{-singular}\Longrightarrow \pi_{I'}(\bx) \text{ is }\pi_{I'}(F_{\ba}^+)\text{-singular}.
\end{equation*}
We adopt the convention that $\pi_{\varnothing}(\bx)$ is never $\pi_{\varnothing}(F_{\ba}^+)$-singular. This motivates the following definition. 

\begin{definition}\label{D:rayess}
	Let $(F_\ba^+,X)$ be given as in \eqref{E:notation} and \eqref{E:flow}, and $s'\in \{0,1,\cdots,s-1\}$. A point $\bx\in X$ is called \textsl{$s'$-essentially $F_{\ba}^+$-singular} if it is $F_{\ba}^+$-singular, but for any $I\in \cI_{s'}$, $\pi_I(\bx)$ is not $\pi_I(F_{\ba}^+)$-singular.
\end{definition}

In particular, the definition of essential singularity in \cite{AGMS} coincides with our definition of $(s-1)$-essential singularity. For $s'\in \{0,1,\cdots,s-1\}$, we denote by $D^{e,s'}(F_{\ba}^+,X)$ the set of $s'$-essentially $F_{\ba}^+$-singular points in $X$. Then it is clear that $$
D^{e,s-1}(F_{\ba}^+,X)\subseteq \cdots\subseteq D^{e,1}(F_{\ba}^+,X)\subseteq D^{e,0}(F_{\ba}^+,X)=D(F_{\ba}^+,X).
$$
The paper \cite{AGMS} actually showed the following stronger result than Theorem \ref{T:AGMS}: \begin{theorem}\label{T:eAGMS}
	Let $(F_\ba^+,X)$ be given as in \eqref{E:notation} and \eqref{E:flow}, and $s'\in \{0,1,\cdots,s-1\}$. Then 
	\begin{equation}
		\dim D^{e,s'}(F^+_\ba, X)=\dim X- \min_{1\le i\le s} \frac{m_in_i}{m_i+n_i}.
	\end{equation}
\end{theorem}

\subsection{Essential singularity for cone actions} 
The goal of this paper is to obtain an analogue of Theorem \ref{T:eAGMS} when the ray $F_\ba^+$ is replaced by a cone. 

The study of dynamical behavior for cone actions plays a central role in the field of homogeneous dynamics. For example, the famous Littlewood's conjecture can be reformulated as classifying bounded orbits of certain cone actions on $\SL_3(\RR)/\SL_3(\ZZ)$ (see e.g. \cite[Proposition 11.1]{EKL06} for details). In \cite{Weiss}, Weiss introduced an \textit{obvious} type of divergent orbits for cone actions, and showed that there are non-obvious divergent orbits for cone actions when the cone is contained in a positive Weyl chamber. This was further generalized by Tamam in \cite{Tamam21}. 


From now on, we keep the setting \eqref{E:notation} and consider the multi-parameter diagonal subsemigroup of $G$: $$
A^+:=\prod\limits_{i=1}^sF_{m_i,n_i}^+.
$$
A \textsl{cone} $C\subseteq A^+$ is defined to be a closed subsemigroup of $A^+$. We first extend the concepts of singularity and essential singularity to cone actions on $X$.

\begin{definition}\label{D:coneess}
	Let $X$ be given as in \eqref{E:notation}, and let $C\subseteq A^+$ be a cone. A point $\bx\in X$ is called \emph{$C$-singular} if for any compact subset $K\subset X$, the set $\{c\in C: c\bx\in K\}$ is compact.
	
	Let $s'\in \{0,1,\cdots,s-1\}$. A point $\bx\in X$ is called \emph{$s'$-essentially $C$-singular} if it is $C$-singular, but for any $I\in \cI_{s'}$, $\pi_{I}(\bx)$ is not $\pi_I(C)$-singular.
\end{definition}

In particular, when the cone $C$ degenerates to a ray $F_{\ba}^+\,(\ba\in\RR^s_+)$, the definition of essential $C$-singularity here coincides with the definition of essential $F_{\ba}^+$-singularity before. For $s'\in \{0,1,\cdots,s-1\}$, we denote by $D^{e,s'}(C,X)$ the set of $s'$-essentially $C$-singular points in $X$. Then it is clear that $$
D^{e,s-1}(C,X)\subseteq \cdots\subseteq D^{e,1}(C,X)\subseteq D^{e,0}(C,X)=D(C,X).
$$
Note that if $x\in X$ is $(s-1)$-essentially $C$-singular, then the divergent orbit $Cx$ is non-obvious in the sense of \cite[Definition 4.1]{Weiss}.



Let $C\subseteq A^+$ be a cone. Our main focuses in this paper will be \begin{itemize}
	\item the upper bound of $\dim{D^{e,1}(C,X)}$;
	\item the lower bound of $\dim{D^{e,s-1}(C,X)}$.
\end{itemize}

To state our main results, we first parametrize a cone $C\subseteq A^+$ in an explicit way. For any nonempty closed convex subset $V\subset \RR^s_+$, set $$C_V:=\{g_{\ba t}: \ba\in V,\; t\geq 0\}.$$ 
It is clear that $C_V\subseteq A^+$ is a cone. For $\ba=(a_1,\cdots,a_s)\in\RR^s_+$ and $\epsilon>0$, set $$
V(\ba,\epsilon):=\left\{
\bd=(d_1,\cdots,d_s)\in \RR^s_+:1-\epsilon\leq \frac{a_i}{d_i}\leq 1+\epsilon,\;\forall 1\leq i\leq s
\right\}.
$$

Our main results in this paper are the following:

\begin{theorem}\label{T:mainupper}
	Let $\ba=(a_1,\cdots,a_s)\in \RR^s_+$. Then for any $\epsilon\in (0,1)$,  $$
	\dim D^{e,1}(C_{V(\ba,\epsilon)},X)\leq \dim{X}-\min_{1\leq i\leq s}\frac{m_in_i}{m_i+n_i}-\frac{2}{3}\epsilon\cdot \min_{1\leq i\leq s}\frac{m_in_i}{m_i+n_i}.
	$$
\end{theorem}

\begin{theorem}\label{T:mainlower}
	Let $\ba=(a_1, \ldots, a_s)\in \RR^s_+$. Then for any $\epsilon\in (0, 1/64)$, 
	\[\dim D^{e,s-1}(C_{V(\ba, \epsilon)}, X) \ge \dim{X}- \min_{1\le i\le s} \frac{m_in_i}{m_i+n_i}-8\epsilon \left(\sum_{i=1}^s\frac{m_in_i}{m_i+n_i}-\min_{1\leq i\leq s}\frac{m_in_i}{m_i+n_i}\right).\]
\end{theorem}

Several important remarks are in order: \begin{enumerate}
    \item Theorems \ref{T:mainupper} and \ref{T:mainlower} are new even for $X=\big(\SL(2,\RR)/\SL(2,\ZZ)\big)^s$, namely, when $(m_i, n_i)=(1, 1)$ for all $1\le i\le s$. All key ingredients of the proofs already appear in this simplest situation. 
	\item For the cone $C=C_{V(\ba,\epsilon)}$, both the upper bound of $\dim{D^{e,1}(C,X)}$ and the lower bound of $\dim{D^{e,s-1}(C,X)}$ are independent of the center $\ba\in\RR^s_+$; both of them are linear with respect to the angular aperture $\epsilon>0$.
	\item Letting $\epsilon\to 0^+$, we immediately get $$
	\dim D^{e,s-1}(F_{\ba}^+, X)=\dim D^{e,1}(F_{\ba}^+, X)=\dim{X}- \min_{1\le i\le s} \frac{m_in_i}{m_i+n_i}.
	$$
	So Theorems \ref{T:mainupper} and \ref{T:mainlower} are generalizations of Theorem \ref{T:eAGMS} except $s'=0$.
	\item 
	Another extremal case is when $C=A^+$.
	It is not hard to check that $$D(A^+, X)=\prod_{1\le i\le s} D(F_{m_i,n_i}^+, X_i), \quad\text{hence}\quad D^{e,1}(A^+, X)=\emptyset.$$
	So it makes sense to only consider cones of the form $C=C_{V(\ba,\epsilon)}$.
\end{enumerate}

\subsection{Some further expectations}
There are several directions in which our main results could be further extended. Firstly, we pose the following natural problem that seeks to generalize the main theorems of this paper.
\begin{problem}\label{P:gencone}
Let $(G,X)$ be given as in \eqref{E:notation}, and let $C$ be a cone in the full diagonal subgroup of $G$ with angular aperture $\epsilon>0$. Then does an analogous estimate for $\dim{D^e(C,X)}$ still hold? Namely, are there some constants $\kappa_1,\kappa_2>0$ such that for any sufficiently small $\epsilon>0$, $$
\dim{X}-\min_{1\leq i\leq s}\frac{m_in_i}{m_i+n_i}-\kappa_1\epsilon\leq \dim{D^e(C,X)}\leq \dim{X}-\min_{1\leq i\leq s}\frac{m_in_i}{m_i+n_i}-\kappa_2\epsilon\;?
$$
\end{problem}

We expect that the answer to Problem \ref{P:gencone} is affirmative. The reason is as follows. Recall that in our original setting $C\subseteq A^+$, one has $\dim{C}=\dim{A^+}=s\leq \mathrm{rank}{G}$, where the equality holds if and only if $m_i=n_i=1$ for all $1\le i\le s$. From a dynamical point of view, the behavior of actions differs substantially depending on whether the cone $C$ has full rank in $G$; nevertheless, our main theorems provide a unified linear estimate for $\dim{D^e(C,X)}$ in both cases. This supports a positive answer to the problem above.


Furthermore, we would like to pose the following problem for general homogeneous spaces with real rank at least two.
\begin{problem}
Let $n\ge 3$, $G=\SL(n, \RR)$, $\Gamma=\SL(n,\ZZ)$, $X=G/\Gamma$, and let $A$ be the full diagonal subgroup of $G$. Let $F^+=\{g_t:t\geq 0\}$ be a one-parameter subsemigroup in $A$, 
and let $C\subseteq A$ be a cone containing $F^+$ with angular aperture $\epsilon>0$. Then does an analogous estimate for $\dim{D^e(C,X)}$ still hold? Namely, are there some constants $\kappa_1,\kappa_2>0$ such that for any sufficiently small $\epsilon>0$, $$
\dim{D(F^+,X)}-\kappa_1\epsilon\leq \dim{D^e(C,X)}\leq \dim{D(F^+,X)}-\kappa_2\epsilon\;?
$$
\end{problem}

\begin{remark}\label{rmk-barak}
In the above two problems, we only state our expectation for sufficiently small $\epsilon>0$. One key reason is that, as shown by Weiss in \cite{Weiss}, for a cone $C\subseteq A$ that is large enough, the set of $C$-divergent points in the homogeneous space is highly sensitive to 
the delicate algebraic structure of the cone $C$.
\end{remark}

\subsection{Organization of the paper}

The organization of this paper is as follows. In Section \ref{s-reduction}, we reduce Theorems \ref{T:mainupper} and \ref{T:mainlower} to Propositions \ref{P:mainupper} and \ref{P:mainlower} respectively, which give the dimension estimates of the set of essentially singular points on an unstable horospherical leaf.
The proofs of Propositions \ref{P:mainupper} and \ref{P:mainlower} consist of the main body of the paper, and they will be presented in two independent parts.

Section \ref{S:UB} is devoted to the proof of Propositions \ref{P:mainupper}. In this part, we basically follow the lines of \cite[Section 3]{AGMS} to construct a universal covering of $D^{e,1}(C_{V(\ba,\epsilon)},\bM)$ independent of the center $\ba$. Compared with before, our main innovative part here is Lemma \ref{L:mathcover}, which essentially provides the extra $\epsilon$-drop of the dimension. Our method here can also be used to give the upper bound for the dimension of the set of essentially singular points in other product systems.


Section \ref{s-lower} is devoted to the proof of Proposition \ref{P:mainlower}. In this part, we still employ the variational principle in parametric geometry of numbers introduced in \cite{VarPrinc}. So our task is to explicitly construct certain \textsl{templates} satisfying the desired properties. The extra $\epsilon$-drop of the dimension comes from the requirement \eqref{e-ess-dim-2}, which makes the construction and verification much more complicated than before.


{\bf Acknowledgments:} The authors would like to thank Jinpeng An and Barak Weiss for their helpful comments. Indeed, Remark \ref{rmk-barak} is suggested by Barak Weiss.
L. Guan is supported by National Key R\&D Program of China No. 2022YFA1007800, NSFC and Zhejiang Provincial Innovative Resource Allocation Project  2022XHSJJ010.

\section{Essentially joint singularity of matrix tuples}\label{s-reduction}
In this section, we deduce Theorems \ref{T:mainupper} and \ref{T:mainlower} from their counterparts in Diophantine approximation, which concerns ``essentially joint singularity properties'' of matrix tuples.  Let us fix an integer $s\ge 2$, a pair $(m_i, n_i)\in \NN^2$ for each $1\le i\le s$, and denote
\[M_i=M_{m_i\times n_i}(\RR)\quad \text{ and }\quad \bM=\prod_{i=1}^s M_i.\]
For $\bTheta=(\btheta_1,\ldots, \btheta_s)\in \bM$, let
\[u_{\bTheta}=(u_{\btheta_1},\ldots, u_{\btheta_s})\in G \quad\text{ and }\quad x_{\bTheta}=(x_{\btheta_1},\ldots, x_{\btheta_s})\in X, \]
where  $u_\btheta$ and $x_\btheta$ are as in (\ref{eq;utheta}).
It is easily checked that the set $$U:=\{u_{\bTheta}: \bTheta\in \bM\}$$ is the expanding horospherical subgroup of $G$ with respect to any one-parameter subsemigroup $F^+\subseteq A^+$. 

Let $C\subseteq A^+$ be a cone. For $s'\in \{0,1,\cdots,s-1\}$, consider the following set of matrix tuples:
\begin{align*}
D^{e,s'}(C, \bM):=\{\bTheta\in \bM: x_{\bTheta}\in D^{e,s'}(C,X)\}.
\end{align*}
Then Theorems \ref{T:mainupper} and \ref{T:mainlower} can be deduced from the following two propositions respectively.

\begin{proposition}\label{P:mainupper}
	Let $\ba=(a_1,\cdots,a_s)\in \RR^s_+$. Then for any $\epsilon\in (0,1)$, one has $$
	\dim D^{e,1}(C_{V(\ba,\epsilon)},\bM)\leq \sum_{i=1}^sm_in_i-\left(1+\frac{2}{3}\epsilon\right)\min_{1\leq i\leq s}\frac{m_in_i}{m_i+n_i}.
	$$
\end{proposition}

\begin{proposition}\label{P:mainlower}
	Let $\ba=(a_1, \ldots, a_s)\in \RR^s_+$. Then for any $\epsilon\in (0, 1/64)$, one has
	\[\dim D^{e,s-1}(C_{V(\ba, \epsilon)}, \bM) \ge \sum_{i=1}^sm_in_i- \min_{1\le i\le s} \frac{m_in_i}{m_i+n_i}-8\epsilon \left(\sum_{i=1}^s\frac{m_in_i}{m_i+n_i}-\min_{1\leq i\leq s}\frac{m_in_i}{m_i+n_i}\right).\]
\end{proposition}


To do this reduction, we only need to establish the following lemma: 

\begin{lemma}
	Let $C\subseteq A^+$ be a cone, and $s'\in \{0,1,\cdots,s-1\}$. Then $$
	\dim{X}-\dim{D^{e,s'}(C,X)}=\dim{\bM}-\dim{D^{e,s'}(C,\bM)}.
	$$
\end{lemma}

\begin{proof}
Let $P$ be the weakly contracting subgroup of $G$ with respect to $C$, i.e.,
$$P=\left\{h\in G: \text{ the set } \{ghg^{-1}: g\in C\} \text{ is bounded}\right\}.$$
	Then $P$ is a parabolic subgroup of $G$ whose Lie algebra is complementary to the Lie algebra of $U$. It is straightforward to verify that the set $PU:=\{pu: p\in P, u\in U \}$ consists of elements $(g_1,\ldots,g_s)$ in $G$ such that for each $1\le i\le s$, the submatrix of $g_i$ formed by its first $m_i$ rows and first $m_i$ columns is invertible. In particular, $PU$ is Zariski open in $G$.
On the other hand, by Borel's density theorem \cite{B}, every left coset of $\Gamma$ is Zariski dense in $G$.
It follows that the map
$$\pi: P\times \bM\rightarrow X, \quad (p, \bTheta)\mapsto px_{\bTheta}$$
is surjective.

Note that for any $s'\in \{0,1,\cdots,s-1\}$, any $p\in P$ and $x\in X$, if  $px$ is $s'$-essentially $C$-singular, then so is  $x$.
Hence we have
$$\pi^{-1}( D^{e,s'}(C, X))=P\times  D^{e,s'}(C, \bM).$$
Since the multiplication map $P\times U\to PU $ is a diffeomorphism (see, e.g., \cite[Lemma 6.44]{Knapp}), the map  $\pi$ is a local diffeomorphism. Thus we have
\[\dim{X}=\dim{(P\times \bM)},\quad \text{and}\quad
\dim D^{e,s'}(C, X)= \dim \pi^{-1}(D^{e,s'}(C, X)).\]
Note that for any subset $Y$ of $\bM$, $\dim (P\times Y)= \dim P+\dim Y$.
So we conclude that \begin{align*}
	\dim D^{e,s'}(F^+_\ba, X)&=\dim{P}+\dim{D^{e,s'}(F^+_\ba, \bM)}\\
	&=\dim{X}-\dim{M}+\dim{D^{e,s'}(F^+_\ba, \bM)}.
\end{align*}
This is as desired.
\end{proof}

\section{The upper bound estimate}\label{S:UB}
The aim of this section is to establish the upper bound estimate Proposition \ref{P:mainupper}. For the sake of simplicity, in this section we write $D^e=D^{e,1}$. So we will prove the following result:

\begin{proposition}\label{P:uppdim}
	Let $\ba=(a_1,\cdots,a_s)\in \RR^s_+$. Then for any $\epsilon\in (0,1)$, one has $$
	\dim D^e(C_{V(\ba,\epsilon)},\bM)\leq \sum_{i=1}^sm_in_i-\left(1+\frac{2}{3}\epsilon\right)\min_{1\leq i\leq s}\frac{m_in_i}{m_i+n_i}.
	$$
\end{proposition}




\subsection{Reformulation of singular properties and a sketch of proof}

Let us introduce some notations. For $\ba\in\RR^s_+$, $\bTheta\in\bM$, $R>0$, and $1\leq i\leq s$, set \begin{align*}
	\cQ_{\ba,R}(\bTheta)&:=\{t>0:g_{\ba t}x_{\bTheta}\notin B^X_R\};\\
	\cQ_{\ba,R,i}(\bTheta)&:=\{t>0:\pi_i(g_{\ba t}x_{\bTheta})\notin B^{X_i}_R\}.
\end{align*}
It is clear from our choice of metrics that $\cQ_{\ba,R}(\bTheta)=\bigcup_{i=1}^s\cQ_{\ba,R,i}(\bTheta)$.

\begin{lemma}\label{L:singulartail}
	Let $\ba\in\RR^s_+$, $\bTheta\in\bM$, and $\epsilon\in (0,1)$. Then \begin{enumerate}
		\item[(1)] $\bTheta\in D(C_{V(\ba,\epsilon)},\bM)$ if and only if for any $R>0$, there exists some $T_0>0$ such that for any $\bd\in V(\ba,\epsilon)$, the set $\cQ_{\bd,R}(\bTheta)$ contains $(T_0,+\infty)$;
		\item[(2)] $\bTheta\in D^e(C_{V(\ba,\epsilon)},\bM)$ if and only if $\bTheta\in D(C_{V(\ba,\epsilon)},\bM)$, but there exists some $R>0$ such that for any $1\leq i\leq s$ and any $\bd\in \RR^s_+$, the set $\RR_+\setminus\cQ_{\bd,R,i}(\bTheta)$ is unbounded.
	\end{enumerate}
\end{lemma}
\begin{proof}
	(1) By definition, $\bTheta\in D(C_{V(\ba,\epsilon)},\bM)$ means that for any $R>0$, the following set \begin{equation}\label{E:dtset}
		\{\bd t:\bd\in V(\ba,\epsilon),t\geq 0,g_{\bd t}x_{\bTheta}\in B^X_R\} 
	\end{equation}
	is compact. 
	
	For the necessity part, suppose that there exists $R>0$ such that for any $n\geq 1$, one may find some $\bd_n\in V(\ba,\epsilon)$ and some $t_n>n$ with $g_{\bd_n t_n}x_{\bTheta}\in B^X_R$. In particular, the set \eqref{E:dtset} contains an unbounded set $\{\bd_nt_n:n\geq 1\}$, so $\bTheta\notin D(C_{V(\ba,\epsilon)},\bM)$.
	
	For the sufficiency part, suppose that there exists $R>0$ such that the set \eqref{E:dtset} is unbounded. So one may find a sequence $(\bd_n)_{n\geq 1}$ in $V(\ba,\epsilon)$ and a sequence $t_n\to +\infty$ such that for any $n\geq 1$, $g_{\bd_n t_n}x_{\bTheta}\in B^X_R$. In particular, for any given $T_0>0$, we see that when $n\geq 1$ is large enough, $t_n\in (T_0,+\infty)\setminus \cQ_{\bd_n,R}(\bTheta)$. The proof is complete.
	~\\
	(2) By definition, $\bTheta\in D^e(C_{V(\ba,\epsilon)},\bM)$ means that $\bTheta\in D(C_{V(\ba,\epsilon)},\bM)$, while for any $1\leq i\leq s$, one has $\btheta_i\notin D(\pi_i(C_{V(\ba,\epsilon)}),M_i)=D(F_{m_i,n_i}^+,M_i)$. The latter condition just says that there exists some $R>0$ such that for any $1\leq i\leq s$, $$
    \{t>0:g_t^{(m_i,n_i)}x_{\btheta_i}\in B^{X_i}_R\}\text{ is unbounded}.
    $$
    
    Note that for any $\bd\in\RR^s_+$, we have $$
    \{t>0:g_t^{(m_i,n_i)}x_{\btheta_i}\in B^{X_i}_R\}=\frac{1}{d_i}\cdot (\RR_+\setminus \cQ_{\bd,R,i}(\bTheta)).
    $$
    So their unboundedness are equivalent.
\end{proof}


For $\bb\in \RR^s_+$, $\bTheta=(\btheta_1,\cdots,\btheta_s)\in\bM$, $R>0$, and $1\leq i\leq s$, define \begin{align*}
	\cA(F_{\bb}^+,R,\bTheta)&:=\limsup_{T\to +\infty}\frac{1}{T}\sum_{i=1}^s\int_0^T1_{X_i\setminus B^{X_i}_R}(\pi_i(g_{\bb t}x_{\bTheta}))\mathrm{d}t;\\
	\cA_i(F_{\bb}^+,R,\btheta_i)&:=\limsup_{T\to +\infty}\frac{1}{T}\int_0^T1_{X_i\setminus B^{X_i}_R}(\pi_i(g_{\bb t})x_{\btheta_i})\mathrm{d}t.
\end{align*}
It is clear that $\cA(F_{\bb}^+,R,\bTheta)\leq \sum_{i=1}^s \cA_i(F_{\bb}^+,R,\btheta_i)$.

\begin{lemma}\label{L:abchange}
	Let $\ba,\bb\in\RR^s_+$ and $\epsilon\in (0,1)$. Then for any $R>0$, $$
	D^e(C_{V(\ba,\epsilon)},\bM)\subseteq \left\{\bTheta\in\bM:\cA(F_{\bb}^+,R,\bTheta)>1+\frac{2}{3}\epsilon\right\}.
	$$
\end{lemma}

The proof of Lemma \ref{L:abchange} will be given in Section \ref{SS:abchange}. In the following, we shall choose $\bb\in\RR^s_+$ to be the special weight 
\begin{equation}\label{E:b0}
	{\mathbf b}_0=\left(\frac{m_1n_1}{m_1+n_1},\ldots,\frac{m_sn_s}{m_s+n_s}\right).
\end{equation}
In this case the dynamical system $(F^+_{{\mathbf b}_0}, X)$ has a single positive Lyapunov exponent.
By Lemma \ref{L:abchange}, one has
\begin{align}\label{E:abinc}
	D^e(C_{V(\ba,\epsilon)},\bM)\subseteq D_{\epsilon}:=\bigcap_{R>0}\left\{
	\bTheta\in\bM:\cA(F_{\bb_0}^+,R,\bTheta)>1+\frac{2}{3}\epsilon
	\right\}.
\end{align}
So Proposition \ref{P:uppdim} will follow from the next lemma.

\begin{lemma}\label{L:dimuppeps}
	Let $\epsilon\in (0,1)$. Then \begin{equation}
		\dim{D_{\epsilon}}\leq \sum_{i=1}^sm_in_i-\left(1+\frac{2}{3}\epsilon\right)\min_{1\leq i\leq s}\frac{m_in_i}{m_i+n_i}.
	\end{equation}
\end{lemma}

The proof of Lemma \ref{L:dimuppeps} will be given in Section \ref{SS:dimuppeps}. Since the right hand side of (\ref{E:abinc}) does not depend on $\ba\in \RR^s_+$, Lemma
\ref{L:dimuppeps} also implies
\begin{equation*}
	\dim \left (\bigcup_{\ba\in \RR_+^s} D^e(C_{V(\ba,\epsilon)}, \bM)\right)\le \sum_{i=1}^s m_in_i-\left(1+\frac{2}{3}\epsilon\right)\min_{1\le i\le s} \frac{m_in_i}{m_i+n_i}.
\end{equation*}

\subsection{Proof of Lemma \ref{L:abchange}}\label{SS:abchange}
The proof of Lemma \ref{L:abchange} is based on the following key lemma, which is a highly nontrivial extension of \cite[Lemma 3.5]{AGMS}.

\begin{lemma}\label{L:mathcover}
	Let $J_i\subset \RR_+ (1\le i\le s)$ be open subsets satisfying that for each $i$, the set $\RR_+\setminus J_i$ is unbounded. Assume that there exists $T_0\ge 0,\rho_i>0\,(1\leq i\leq s)$ and $\epsilon\in (0,1)$ with 
	\begin{equation}\label{E:assump0}
		(T_0, \infty)\subset \bigcup_{1\le i\le s} \delta_i J_i \quad \text{for all} \quad \delta_i\in (\rho_i(1-\epsilon), \rho_i(1+\epsilon))\ \  (1\le i\le s).
	\end{equation}
	Then we have
	\[ \limsup_{T\to +\infty} \frac{\sum_{1\le i\le s}|J_i\cap [0, T]|}{T}>1+\frac{2}{3}\epsilon.\]
\end{lemma}

\begin{proof}
  To begin, set $A_i=\log J_i$ and $\beta_i=\log \delta_i$. Then 
	by assumption,  for each $i$, the set $\RR\setminus A_i$ is unbounded, and there exist $t_0\in\RR,$ and $\epsilon\in (0,1)$ with \begin{equation}\label{E:assump1}
		(t_0,+\infty)\subset \bigcup_{1\leq i\leq s}(\beta_i+r_i+A_i) \, \text{ for all } r_1,\cdots,r_s\in (\log(1-\epsilon),\log(1+\epsilon)). 
	\end{equation}
	We need to show that $$
	\limsup_{t\to +\infty}\int_{-\infty}^t\sum_{i=1}^s1_{A_i}(u)e^{u-t}\mathrm{d}u> 1+\frac{2}{3}\epsilon.
	$$
	The proof is divided into several steps.
    \vspace{1em}
	
	\textsl{Step 1: Construction of a new cover.} 

    In Step 1, we construct a new cover of a right half-line using the assumption \eqref{E:assump1}. Our goal is to absorb different translates $r_1,\cdots,r_s$ into a single uniform one.
	
	Fix any $\kappa\in (0,\frac{1}{2})$, write $h:=\min\{\kappa\log(1+\epsilon)-\kappa\log(1-\epsilon),1\}>0$. Define $$
	D_i:=\{u\in\RR:[u-\beta_i-h,u-\beta_i+h]\subset A_i\}\,(1\leq i\leq s).
	$$
	We first claim that: the assumption \eqref{E:assump1} implies that for $t_0'=t_0-\frac{1}{2}\log(1-\epsilon^2)$, \begin{equation}\label{E:Dcover}
		(t_0',+\infty)\subset\bigcup_{1\leq i\leq s}D_i.
	\end{equation}
	In fact, fix any $u>t_0$. We see from \eqref{E:assump1} that for any $r_1,\cdots,r_s\in (\log(1-\epsilon),\log(1+\epsilon))$, there exists $1\leq i\leq s$ such that $u-\beta_i-r_i\in A_i$. It follows that there exists $1\leq i\leq s$ such that for any $r\in (\log(1-\epsilon),\log(1+\epsilon))$, one has $u-\beta_i-r\in A_i$. So we obtain $$
	(t_0,+\infty)\subseteq \bigcup_{1\leq i\leq s}\{u\in\RR:(u-\beta_i-\log(1+\epsilon),u-\beta_i-\log(1-\epsilon))\subseteq A_i\}.
	$$
	Then the claim \eqref{E:Dcover} follows from this.
	
	Note that each $D_i\,(1\leq i\leq s)$ is an open subset and that no $D_i$ contains a right half-line. This completes the construction of a new cover.\vspace{1em}
    
    \textsl{Step 2: Disjointness of intervals of the same color.}
    
    In Step 2, our task is to construct a sequence of intervals inside $A_i\,(1\leq i\leq s)$ satisfying certain disjoint properties. This is the most tricky part in the proof.
    
    We first construct a sequence of real numbers $a_0<a_1<\cdots$ with $a_k\to +\infty$ and a sequence of indices $i_0,i_1,\cdots$ in $\{1,\cdots,s\}$ as follows.
	
	Pick any $a_0>t_0'$. Now for $k\geq 0$, suppose that $a_k$ is chosen. We need to choose the index $i_k$ and the next point $a_{k+1}$. By \eqref{E:Dcover}, the following set is nonempty: $$
	\cI_k:=\{i\in\{1,\cdots,s\}:a_k\in D_i\}.
	$$
	For each $i\in \cI_k$, write $$
	R_i(a_k):=\sup\{b>a:[a_k,b)\subseteq D_i\}.
	$$
	It follows that $a_k<R_i(a_k)<+\infty$ by the properties of $D_i$. Then we choose the index $i_k\in \cI_k$ such that $$
	R_{i_k}(a_k)=\max_{i\in \cI_k}R_i(a_k),
	$$
	and set $a_{k+1}=R_{i_k}(a_k)$. By induction, the construction process is complete, and it remains to show that $a_k\to +\infty$.
	
	In fact, suppose that $a_k\not\to +\infty$. Since $a_k<a_{k+1}$ for $k\geq 0$, we see that the limit $L=\lim_{k\to +\infty}a_k$ exists. Since $L>t_0'$, there exists some $1\leq i\leq s$ with $L\in D_i$. Then there exists some $\eta>0$ such that $(L-\eta,L+\eta)\subset D_i$. When $k$ is large enough, we have $a_k\in (L-\eta,L)$ and hence that $[a_k,L+\eta)\subseteq D_i$. By our construction, $$
	a_{k+1}=R_{i_k}(a_k)\geq R_i(a_k)\geq L+\eta,
	$$
	which is a contradiction!
	
	Now for each $k\geq 0$, write $$
	I_k=[a_k-\beta_{i_k}-h,a_{k+1}-\beta_{i_k}+h).
	$$
	Since $[a_k,a_{k+1})\subset D_{i_k}$, we see that $I_k\subseteq A_{i_k}$. The length of $I_k$ is $(a_{k+1}-a_k)+2h$. We claim that: \begin{equation}\label{E:diffcolor}
		i_k=i_l \text{ and }k<l\Longrightarrow I_k\cap I_l=\varnothing.
	\end{equation}
	
	Suppose that $i_k=i_l=i$ for some $k<l$ but $I_k\cap I_l\neq \varnothing$. 
	It follows that $I_k\cup I_l=[a_k-\beta_i-h,a_{l+1}-\beta_i+h)$ is contained in $A_i$ and hence that $[a_k,a_{l+1})\subseteq D_i$. By the maximality of $a_{k+1}$, we have $a_{l+1}\leq a_{k+1}$, which is a contradiction. \vspace{1em}

    \textsl{Step 3: A lower bound of $\sum_{i=1}^s1_{A_i}$ and a key integral of $\sum_{k}1_{I_k}$.}

    In Step 3, our goal is to give a lower bound estimate of $\sum_{i=1}^s1_{A_i}$ using the intervals $\{I_k\}_{k\geq 0}$ constructed in Step 2. We also introduce a key integral of $\sum_{k}1_{I_k}$ over certain interval in $\RR$, which is crucial to the proof in Step 4.
    
	We consider the function $$
	m(u)=\sum_{i=1}^s1_{A_i}(u).
	$$
	For any block $K,K+1,\cdots,K+N-1$ where $K,N\geq 1$, write $$
	M(u)=\sum_{k=K}^{K+N-1}1_{I_k}(u). 
	$$
	We show that $M(u)\leq m(u)$ for any $u\in\RR$. In fact, for any $u\in\RR$, we have $$
	M(u)=\sum_{i=1}^s\sum_{\substack{i_k=i\\ K\leq k\leq K+N-1}}1_{I_k}(u)
	\overset{\text{\eqref{E:diffcolor}}}{\leq}\sum_{i=1}^s 1_{A_i}(u)=m(u).
	$$
	
	Moreover, note that all intervals $I_K,I_{K+1},\cdots,I_{K+N-1}$ are contained in $$
	W:=[a_K-\beta_{\max}-h,a_{K+N}-\beta_{\min}+h],
	$$
	where $\beta_{\min}=\min_{1\leq i\leq s}\beta_i,\;\beta_{\max}=\max_{1\leq i\leq s}\beta_i$. We calculate that \begin{equation}\label{E:int}
		\begin{aligned}
			\int_W(M(u)-1)\mathrm{d}u&=\sum_{k=K}^{K+N-1}|I_k|-|W|\\
			&=((a_{K+N}-a_K)+2Nh)-((a_{K+N}-a_K)+(\beta_{\max}-\beta_{\min})+2h)\\
			&=2(N-1)h-(\beta_{\max}-\beta_{\min}).
		\end{aligned}
	\end{equation}
    This completes the computation in Step 3.\vspace{1em}
	
	\textsl{Step 4: Conclusion of the proof modulo an upper estimate for the key integral.}

    In Step 4, we first assume the truthness of an upper estimate for the key integral $\int_W(M(u)-1)\mathrm{d}u$, and then deduce our conclusion from this.
	
	Fix any $\eta\in (0,1)$. We first claim that: if \begin{equation}\label{E:leqeta}
		\limsup_{t\to+\infty}\int_{-\infty}^tm(u)e^{u-t}\mathrm{d}u<1+\eta,
	\end{equation}
	then \begin{equation}\label{E:estiint}
		\int_W(M(u)-1)\mathrm{d}u\leq 1+(N-1)\cdot \frac{\eta}{1-\eta}.
	\end{equation}

    We assume the truthness of this claim for the time being. Let us choose $\eta\in (0,1)$ such that $2h>\frac{\eta}{1-\eta}$. Suppose that \eqref{E:leqeta} holds. Then a combination of \eqref{E:int} and \eqref{E:estiint} gives for any $N\geq 1$, $$
	2(N-1)h-(\beta_{\max}-\beta_{\min})\leq 1+(N-1)\cdot \frac{\eta}{1-\eta}.
	$$
	Letting $N\to +\infty$ yields $2h\leq \frac{\eta}{1-\eta}$, which is a contradiction. So we must have $$
	\limsup_{t\to+\infty}\int_{-\infty}^tm(u)e^{u-t}\mathrm{d}u\geq 1+\eta
	\quad \text{ for any }\eta\in (0,1)\text{ with }2h>\frac{\eta}{1-\eta}.
	$$
	Therefore, $$
	\limsup_{t\to+\infty}\int_{-\infty}^tm(u)e^{u-t}\mathrm{d}u
	\geq 1+\frac{2h}{2h+1}\geq 1+\min\left\{
	\frac{4\kappa\epsilon}{4\kappa\epsilon+1},\frac{2}{3}
	\right\}.$$
	This yields our conclusion by choosing $\kappa$ close to $\frac{1}{2}$.
    \vspace{1em}

    \textsl{Step 5: Verification of the claim in Step 4.}
    
    In Step 5, we verify the claim in Step 4 and hence completes the whole proof. 
    
	Suppose that \eqref{E:leqeta} holds. Write $$
	F(t)=\int_{-\infty}^tm(u)e^{u-t}\mathrm{d}u.
	$$
	Then $F(\cdot)$ is absolutely continuous and in particular almost everywhere differentiable; moreover, for $t$ large enough, one has $F(t)< 1+\eta$. We choose $K>1$ large enough such that 
	$$t\in \bigcup_{k=K}^{K+N-1}I_k\Longrightarrow F(t)< 1+\eta.
	$$
	It is straightforward to see that for almost every $t\in\RR$, $$
	F'(t)=m(t)-F(t)\geq M(t)-F(t).
	$$ 
	
	Here we split the integral $\int_W(M(u)-1)$ into two parts: $$
	\int_W(M(u)-1)\mathrm{d}u=O_1-O_2,\quad \text{ where }O_1=\int_W(M(u)-1)_+\mathrm{d}u,\quad O_2=\int_W(1-M(u))_+\mathrm{d}u.
	$$
	Write $E:=W\cap \{F<1\}$. We further split the integral $O_1$ into two parts: $$
	O_1=O_{11}+O_{12},\quad \text{ where }O_{11}=\int_{W\setminus E}(M(u)-1)_+\mathrm{d}u,\quad O_{12}=\int_E(M(u)-1)_+\mathrm{d}u.
	$$

	We first show that: 
	\begin{equation}\label{E:estiO1}
		O_{11}\leq (N-1)\cdot \frac{\eta}{1-\eta}.
	\end{equation}
	In fact, since the function $M(\cdot)$ takes values in $\ZZ$ and is supported in $W$, we have $$
	O_{11}=\int_{\{M\geq 2,F\geq 1\}} (M(u)-1)\mathrm{d}u.
	$$
	Note that the set $\{u\in\RR:M(u)\geq 2\}$ has at most $N-1$ connected components. We write $$
	\{u\in\RR:M(u)\geq 2\}=C_1\sqcup\cdots\sqcup C_L,
	$$
	where $C_1,\cdots,C_L\,(L\leq N-1)$ are intervals. It follows that for any $1\leq l\leq L$ and almost every $u\in C_l$, one has $$
	F'(u)\geq M(u)-F(u)>M(u)-1-\eta\geq 1-\eta>0.
	$$
	In particular, $F$ is increasing on $C_l$. It follows that each $C_l\cap \{F\geq 1\}$ is an interval. Then we have \begin{equation*}
		\begin{aligned}
			\int_{C_l\cap\{F\geq 1\}}(M(u)-1)\mathrm{d}u&\leq \int_{C_l\cap\{F\geq 1\}}(F'(u)+\eta)\mathrm{d}u\\
			&\leq (1+\eta-1)+\eta\cdot |C_l\cap\{F\geq 1\}|\\
			&\leq \eta+\eta\cdot \int_{C_l\cap\{F\geq 1\}}(M(u)-1)\mathrm{d}u.
		\end{aligned}
	\end{equation*}
	This implies that $$
	\int_{C_l\cap\{F\geq 1\}}(M(u)-1)\mathrm{d}u\leq \frac{\eta}{1-\eta}.
	$$
	Therefore, we conclude that $$
	O_{11}=\sum_{l=1}^L\int_{C_l\cap\{F\geq 1\}}(M(u)-1)\mathrm{d}u\leq (N-1)\cdot \frac{\eta}{1-\eta}.
	$$
	This verifies \eqref{E:estiO1}.

	We next show that:
	\begin{equation}\label{E:estiO2}
		O_{12}\leq 1+O_2.
	\end{equation}
	In fact, for almost every $u\in E=W\cap\{F<1\}$,  we have $$
	F'(u)\geq M(u)-F(u)>M(u)-1=(M(u)-1)_+-(1-M(u))_+.
	$$
	So integrating over $E$ gives $$
	1\geq \int_{E}F'(u)\mathrm{d}u\geq \int_{E}(M(u)-1)_+\mathrm{d}u-\int_{E}(1-M(u))_+\mathrm{d}u\geq O_{12}-O_2.
	$$
	This verifies \eqref{E:estiO2}.
	
	In summary, combining \eqref{E:estiO1} and \eqref{E:estiO2} gives the desired conclusion \eqref{E:estiint}. This completes the proof.
\end{proof}

\begin{remark}
The current proof of Lemma \ref{L:mathcover} is based on some conversations with GPT Pro 5.5. Indeed, without the help of AI tools, we can prove a slightly weaker linear bound, in which the coefficient of $\epsilon$ may depend on the center $\ba$.
\end{remark}

Now we turn to the proof of Lemma \ref{L:abchange}.

\begin{proof}[Proof of Lemma \ref{L:abchange}]
	Let $\ba,\bb\in\RR^s_+$, and $\bTheta\in\bM$. For $R>0$ and $1\leq i\leq s$, write $J_i(R):=\cQ_{\bb,R,i}(\bTheta)$. It follows from the definition that each $J_i(R)$ is an open subset of $\RR_+$. Moreover, for any $1\leq i\leq s$ and $\bd\in V(\ba,\epsilon)$, we have $$
	\cQ_{\dd,R,i}(\bTheta)=\frac{b_i}{d_i}\cdot J_i(R).
	$$
	By Lemma \ref{L:singulartail}, 
    when $\bTheta\in D^e(C_{V(\ba,\epsilon)},\bM)$, we have \begin{enumerate}
		\item for any $R>0$, there exists some $T_0(R,\bTheta)>0$ such that for any $\bd\in V(\ba,\epsilon)$, $$
		\bigcup_{i=1}^s\frac{b_i}{d_i}\cdot J_i(R)\supseteq (T_0,+\infty);
		$$
		\item there exists some $R_1=R_1(\bTheta)>0$ such that for any $1\leq i\leq s$, the set $\RR_+\setminus J_i(R_1)$ is unbounded.
	\end{enumerate}
	Recall that the set $\RR_+\setminus J_i(R)$ gets larger as we increase $R$. In particular, for any $R>R_1$ and any $1\leq i\leq s$, the set $\RR_+\setminus J_i(R)$ is unbounded.
	
	Now let $R>R_1$ and $T_0=T_0(R,\bTheta)$. For $1\leq i\leq s$, set $\rho_i=\frac{b_i}{a_i}$. It is readily checked that the assumption \eqref{E:assump0} in Lemma \ref{L:mathcover} is satisfied. So we have 
	\begin{equation}\label{E:keyinq}
		\limsup_{T\to +\infty} \frac{\sum_{1\le i\le s}|J_i(R)\cap [0, T]|}{T}>1+\frac{2}{3}\epsilon.
	\end{equation} 
	Since $J_i(R)$ gets smaller as we increase $R$, the inequality \eqref{E:keyinq} actually holds for any $R>0$. So we conclude that for any $R>0$, $$
	\cA(F_{\bb}^+,R,\bTheta)=\limsup_{T\to +\infty} \frac{\sum_{1\le i\le s}|J_i(R)\cap [0, T]|}{T}>1+\frac{2}{3}\epsilon.
	$$
	This is as desired.
\end{proof}

\subsection{Proof of Lemma \ref{L:dimuppeps}}\label{SS:dimuppeps}
Following the line of \cite[Section 3.3]{AGMS}, we divide the proof of Lemma \ref{L:dimuppeps} into several paragraphs. For simplicity, we write
$$b_i=\frac{m_in_i}{m_i+n_i}, \qquad 1\le i\le s.$$
Then $\bb_0=(b_1,\ldots,b_s)$. Without loss of generality, assume that
\begin{equation*}
	b_1=\min_{1\le i\le s} b_i.
\end{equation*}

\subsubsection{A reduction of the limsup set}

For $\bTheta=(\btheta_1,\cdots,\btheta_s)\in\bM$, $R,T>0$, and $1\leq i\leq s$, we write \begin{align*}
	\cA(F^+_{\bb_0},R,T,\bTheta)&:=\frac{1}{T}\sum_{i=1}^s\int_0^T1_{X_i\setminus B^{X_i}_R}(\pi_i(g_{\bb_0 t}x_{\bTheta}))\mathrm{d}t;\\
	\cA_i(F^+_{\bb_0},R,T,\btheta_i)&:=\frac{1}{T}\int_0^T1_{X_i\setminus B^{X_i}_R}(\pi_i(g_{\bb_0 t})x_{\btheta_i})\mathrm{d}t.
\end{align*}
It is clear that $\cA(F^+_{\bb_0},R,T,\bTheta)=\sum_{i=1}^s\cA_i(F^+_{\bb_0},R,T,\btheta_i)$.
Recall that \begin{align*}
	D_{\epsilon}&=\bigcap_{R>0}\left\{
	\bTheta\in\bM:\cA(F_{\bb_0}^+,R,\bTheta)>1+\frac{2}{3}\epsilon
	\right\}\\
	&=\bigcap_{R>0}\bigcap_{T_0>0}\bigcup_{T>T_0}\left\{
	\bTheta\in\bM:\cA(F^+_{\bb_0},R,T,\bTheta)>1+\frac{2}{3}\epsilon
	\right\}.
\end{align*}
We first observe a simple fact which reduces the limsup for \textsl{all} $T\to +\infty$ to the limsup for certain $T\to +\infty$ along a sequence with bounded gaps. 

\begin{lemma}\label{L:bddgap}
	Let $\epsilon\in (0,1)$. Then for any $R,T>0$, \begin{equation*}
		D_{\epsilon}\subseteq \bigcap_{\ell_0\in\NN}\bigcup_{\ell\in\NN,\;\ell>\ell_0}\left\{
		\bTheta\in\bM:\cA(F^+_{\bb_0},R,\ell T,\bTheta)>1+\frac{2}{3}\epsilon
		\right\}.
	\end{equation*}
\end{lemma}
\begin{proof}
	Let $R,T>0$. It suffices to show that: there exists $R_1>R$ such that for any $l\in\NN$, any $T'\in [(\ell-1)T,\ell T]$, and any $\bTheta\in\bM$, \begin{equation}\label{E:boundedgap}
	\cA(F^+_{\bb_0},R_1,T',\bTheta)\leq
	\max\{\cA(F^+_{\bb_0},R,(\ell-1)T,\bTheta),\cA(F^+_{\bb_0},R,\ell T,\bTheta)\}.
\end{equation}

In fact, given $R>0$, by continuity we may find some $R_1>R$ depending on $R,T$ such that for any $1\leq i\leq s$, \begin{equation}\label{E:defR1}
\{\pi_i(g_{\bb_0t})\cdot x_i:|t|\leq T,x_i\in X_i\setminus B^{X_i}_{R_1}\}\subseteq X_i\setminus B^{X_i}_{R}.
\end{equation}
Now let $\ell\in\NN$, $T'\in [(\ell-1)T,\ell T]$, and $\bTheta\in\bM$. Consider the function 
$$
f:[(\ell-1)T,T']\to\NN\cup\{0\},\quad t\mapsto \sum_{i=1}^s1_{X_i\setminus B^{X_i}_{R_1}}(\pi_i(g_{\bb_0 t}x_{\bTheta})).
$$
There exists some $T_1'\in [(\ell-1)T,T']\subseteq [(\ell-1)T,\ell T]$ such that 
$$
f(T_1')=\max_{t\in [(\ell-1)T,T']}f(t).
$$
So we have \begin{equation}\label{E:esti1}
	\begin{aligned}
	\cA(F^+_{\bb_0},R_1,T',\bTheta)&=\frac{1}{T'}\left(\int_0^{(\ell-1)T}+\int_{(\ell-1)T}^{T'}\right)f(t)\mathrm{d}t\\
	&\leq \frac{(\ell-1)T}{T'}\cA(F^+_{\bb_0},R_1,(\ell-1)T,\bTheta)+\frac{T'-(\ell-1)T}{T'}f(T_1')\\
	&\leq \frac{(\ell-1)T}{T'}\cA(F^+_{\bb_0},R,(\ell-1)T,\bTheta)+\frac{T'-(\ell-1)T}{T'}f(T_1').
	\end{aligned}
\end{equation}

For $1\leq i\leq s$, write $x_i:=\pi_i(g_{\bb_0 T_1'}x_{\bTheta})$. Set $I=\{1\leq i\leq s:x_i\in X_i\setminus B^{X_i}_{R_1}\}$. It is clear that
\begin{equation*}
f(T_1')=\sum_{i=1}^s1_{X_i\setminus B^{X_i}_{R_1}}(x_i)=\#I.
\end{equation*}
For any $i\in I$, we see from \eqref{E:defR1} that $$
\{\pi_i(g_{\bb_0 t}x_{\bTheta}):t\in [(\ell-1)T,\ell T]\}\subseteq \{\pi_i(g_{\bb_0 t})\cdot x_i:|t|\leq T\}\subseteq X_i\setminus B^{X_i}_R.
$$
So we have \begin{equation}\label{E:esti2}
	\begin{aligned}
	\cA(F^+_{\bb_0},R,\ell T,\bTheta)&=\frac{1}{\ell T}\sum_{i=1}^s\left(\int_0^{T'}+\int_{T'}^{\ell T}\right)1_{X_i\setminus B^{X_i}_R}(\pi_i(g_{\bb_0 t}x_{\bTheta}))\mathrm{d}t\\
	&\geq \frac{T'}{\ell T}\cA(F^+_{\bb_0},R,T',\bTheta)+\frac{\ell T-T'}{\ell T}\#I\\
	&\geq \frac{T'}{\ell T}\cA(F^+_{\bb_0},R_1,T',\bTheta)+\frac{\ell T-T'}{\ell T}f(T_1').
\end{aligned}
\end{equation}
Combining \eqref{E:esti1} and \eqref{E:esti2} gives \begin{align*}
\cA(F^+_{\bb_0},R_1,T',\bTheta)&\leq \frac{(\ell-1)(\ell T-T')}{T'}\cA(F^+_{\bb_0},R,(\ell-1)T,\bTheta)\\
&\quad+\frac{\ell(T'-(\ell-1)T)}{T'}\cA(F^+_{\bb_0},R,\ell T,\bTheta)\\
&\leq \max\{\cA(F^+_{\bb_0},R,(\ell-1)T,\bTheta),\cA(F^+_{\bb_0},R,\ell T,\bTheta)\}.
\end{align*}
This verifies \eqref{E:boundedgap} and hence completes the proof.
\end{proof}

\subsubsection{An approximation by product sets}

For $\delta>0$, $R,T>0$, $1\leq i\leq s$, write \begin{align*}
	D_{\delta}(R,T)&:=\left\{\bTheta\in \bM: \cA(F_{\bb_0}^+,R,T,\bTheta)\geq\delta\right\};\\
	D^{(i)}_{\delta}(R,T)&:=\left\{\btheta\in \bM_i: \cA_i(F_{\bb_0}^+,R,T,\btheta)\geq\delta\right\}.
\end{align*}
We shall prove the following simple lemma which approximates $D_{\delta}(R,T)$
by a finite union of product sets.

\begin{lemma}\label{L:prodsets}
	Let $\delta'>0$ and $\delta\in (0,\delta')$. Then there exists a finite subset $\cS=\cS(\delta',\delta)\subseteq [0,\delta]^s$ satisfying the following properties: \begin{enumerate}
		\item[(1)] For any $\bdelta=(\delta_1,\cdots,\delta_s)\in \cS$, $\sum_{i=1}^s\delta_i=\delta$;
		\item[(2)] For any $R,T>0$, $D_{\delta'}(R,T)\subseteq \bigcup_{\bdelta\in\cS}\prod_{i=1}^sD^{(i)}_{\delta_i}(R,T)$.
	\end{enumerate}
\end{lemma}
\begin{proof}
	For any $a>0$, consider the simplex $$
	\Sigma_a:=\left\{\bdelta=(\delta_1,\cdots,\delta_s)\in [0,a]^s:\sum_{i=1}^s\delta_i=a\right\}.
	$$
	Moreover, for $\bdelta=(\delta_1,\cdots,\delta_s)\in \Sigma_{\delta}$, write $$
	N_{\bdelta}:=\{\bdelta'=(\delta_1',\cdots,\delta_s')\in\Sigma_{\delta'}:\delta_i'>\delta_i,\;\forall 1\leq i\leq s\}.
	$$
	Then $\{N_{\bdelta}:\bdelta\in \Sigma_{\delta}\}$ is an open cover of $\Sigma_{\delta'}$. Since $\Sigma_{\delta'}$ is compact, there exists a finite subset $\cS$ of $\Sigma_{\delta}$ such that $\Sigma_{\delta'}=\bigcup_{\bdelta\in\cS}N_{\bdelta}$. Note that for any $R,T>0$, \begin{align*}
		D_{\delta'}(R,T)
		&=\left\{\bTheta=(\btheta_1,\cdots,\btheta_s)\in \bM: \sum_{i=1}^s\cA_i(F_{\bb_0}^+,R,T,\btheta_i)\geq\delta'\right\}\\
		&=\bigcup_{\bdelta'\in\Sigma_{\delta'}}
		\left\{(\btheta_1,\cdots,\btheta_s)\in \bM: \cA_i(F_{\bb_0}^+,R,T,\btheta_i)\geq\delta_i',\;\forall 1\leq i\leq s\right\}\\
		&=\bigcup_{\bdelta\in\cS}\bigcup_{\bdelta'\in N_{\bdelta}}\prod_{i=1}^sD^{(i)}_{\delta_i'}(R,T)\subseteq \bigcup_{\bdelta\in\cS}\prod_{i=1}^sD^{(i)}_{\delta_i}(R,T).
	\end{align*}
	This completes the proof.
\end{proof}

\subsubsection{An application of the covering result}

For the time being, we fix $(m,n)\in \NN^2$. For $r>0$, let $B_r$ denote the open Euclidean ball\footnote{In this paragraph, metric balls in vector spaces are assumed to be open. } in $M_{m\times n}(\RR)$ of radius $r$ centered at $0$. 

Recall that the upper bound part of Theorem \ref{T:eAGMS} in \cite{AGMS} is derived from an adaptation of a covering result in \cite{KKLM}. Here we also need the same technical result. For technical reasons, we include the cases when $\delta\in \{0,1\}$.

\begin{lemma}\label{L:KKLM}
	Let $r>0$. Then there exist $T_0>0$ and a function $\tilde{C}:Y_{m+n}\to\RR_+$ such that the following holds:
	For any $T\ge T_0$, there exists a compact set $\tilde{K}=\tilde{K}(T)$ in $Y_{m+n}$ such that for any $y\in Y_{m+n}$, $\delta\in [0, 1]$ and $\ell \in \NN$, the set
	$$\tilde{Z}_y(r,\tilde{K},\ell T,\delta):=\left\{\btheta\in B_r:\int_{0}^{\ell T} \bone_{Y_{m+n}\setminus \tilde{K}}\left(g_t^{(m,n)}u_\btheta y\right) \dd t\ge \delta\ell T \right\}$$
	can be covered by at most
	$\tilde{C}(y)(\frac{T}{mn})^{3\ell}e^{(m+n-\delta)\ell T}$ balls in $M_{m\times n}(\RR)$ of radius $e^{-\frac{m+n}{mn}\ell T}$.
\end{lemma}
\begin{proof}
	See \cite[Theorem 5.1]{KKLM} and \cite[Corollary 3.7]{AGMS}.
\end{proof}

For $1\le i\le s$, let $\mathrm d_i$ denote the Euclidean metric on $M_i$, and $B_r^{M_i}\subset M_i$ denote the Euclidean ball of radius $r$ centered at $0$. Consider the metric
$$\mathrm d((\btheta_1,\ldots, \btheta_s),(\btheta'_1,\ldots, \btheta'_s))=\max_{1\le i\le s}\mathrm d_i(\btheta_i,\btheta'_i)$$
on $\bM$, and let $B_r^{\bM}=B_r^{M_1}\times\cdots\times B_r^{M_s}$ be the associated metric ball of radius $r$ centered at $0$.
Our covering result for the product space is as follows.

\begin{lemma}\label{L:prodcover}
	Let $\delta'>0$. For any $r,\rho>0$, there exists $T=T(r,\rho)>0$ and $R=R(T)>0$ such that for any $\ell\in\NN$, the set $$
	D_{\delta'}(R,\ell T)\cap B^{\bM}_r
	$$
	can be covered by at most $e^{(\alpha+\rho)\ell T}$ balls of radius $e^{-\ell T}$, where $\alpha:=\sum_{i=1}^sm_in_i-\delta'b_1$.
\end{lemma}
\begin{proof}
	Let $\delta=\delta'-\frac{\rho}{2b_1}$, and $\cS=\cS(\delta',\delta)$ be the finite subset given by Lemma \ref{L:prodsets}. It follows that \begin{equation}\label{E:fincover}
	D_{\delta'}(R,\ell T)\cap B^{\bM}_r\subseteq \bigcup_{\bdelta\in\cS}\prod_{i=1}^s(D^{(i)}_{\delta_i}(R,\ell T)\cap B^{M_i}_r).
\end{equation}
	Let $1\leq i\leq s$. Note that in the notation of Lemma \ref{L:KKLM}, \begin{equation}\label{E:setid}
		D_{\delta_i}^{(i)}(R,\ell T)\cap B^{M_i}_r=\tilde{Z}_{[1_{G_i}]}(r,B^{X_i}_R,b_i\ell T,\delta_i).
	\end{equation}
	In particular, the set \eqref{E:setid} is empty when $\delta_i>1$. So we may assume that $\cS\subseteq [0,1]^s$.

	Let $r>0$. By Lemma \ref{L:KKLM}, there exists $T_i=T_i(r)>0$ and $C_i=\tilde{C}([1_{G_i}])>0$ such that for any $T\geq T_i$, there exists $R_i=R_i(T)>0$ such that for any $\delta_i\in [0,1]$ and $\ell\in\NN$, the set \eqref{E:setid} can be covered by at most $$
	C_i(b_iT/m_in_i)^{3\ell}e^{(m_i+n_i-\delta_i)b_i\ell T}\leq C_iT^{3\ell}e^{(m_in_i-\delta_ib_i)\ell T}
	$$
 	balls of radius $e^{-\ell T}$. Hence for any $T\ge \max_{1\le i\le s}T_i$ and $R \ge \max_{1\le i\le s}R_i(T)$, the set
 	$\prod_{i=1}^s \left(D^{(i)}_{\delta_i}(R, \ell T)\cap  B_r^{M_i}\right)$ can be covered by at most
 	\[C T^{3s\ell}e^{\sum_{i=1}^{s} (m_in_i-\delta_ib_i)\ell T}\]
 	balls of radius $e^{-\ell T}$, where $C=\prod_{i=1}^s C_i$. 
 	
 	Taking $T_0=T_0(\delta',\delta,\rho)>0$ large enough, we may assume that for any $T\ge T_0$ and $\ell\in \NN$,
 	$$\#\cS(\delta',\delta)\cdot C T^{3s\ell}\le e^{\frac{\rho\ell T}{2}}.$$
 	On the other hand, by the choice of $\cS=\cS(\delta',\delta)$, we have
 	\[\sum_{i=1}^{s} (m_in_i-\delta_ib_i)\le \sum_{i=1}^{s} m_in_i-\delta b_1=\alpha+\frac{\rho}{2}.\]
 	In summary, for any $T\ge \max_{0\le i\le s} T_i$ and $R\ge \max_{1\le i\le s} R_i(T)$,
 	the  set \eqref{E:fincover} can be covered by at most
 	\begin{align*}
 		\# \cS(\delta',\delta)\cdot  \max_{\bdelta\in\cS} C T^{3s\ell} e^{\sum_{i=1}^{s} (m_in_i-\delta_ib_i)\ell T}\le e^{(\alpha+\rho)\ell T}
 	\end{align*}
 	balls of radius  $e^{-\ell T}$. This proves the lemma.
\end{proof}

\begin{remark}
	The above argument does not directly extend to a similar covering result for a general weight vector $\ba$. This is the main difficulty in proving Proposition \ref{P:uppdim} and is resolved by Lemma \ref{L:abchange} before.
\end{remark}

\subsubsection{The upper bound of $\dim{D_{\epsilon}}$} Finally, we are prepared to prove Lemma \ref{L:dimuppeps}.

\begin{proof}[Proof of Lemma \ref{L:dimuppeps}]
	Let $\delta'=1+\frac{2}{3}\epsilon$ and $\alpha=\sum_{i=1}^sm_in_i-\delta'b_1$. By the definition of Hausdorff dimension, it suffices to show that for any $r>0$ and $ \sigma>\alpha$, the Hausdorff measure
	\begin{equation*}
		\cH^{ \sigma}( D_{\epsilon}\cap  B_r^{\bM})=0.
	\end{equation*}
	Recall that, for a subset $Z\subset \bM$,
	\begin{equation*}
		\cH^{ \sigma}(Z)=\lim_{\beta\rightarrow 0}\cH^{ \sigma}_\beta(Z),
	\end{equation*}
	where
	\begin{equation*}
		\cH^{ \sigma}_\beta(Z)=\inf \left\{ \sum_k |U_k|^{ \sigma}: Z\subset \bigcup_{k}U_k, |U_k|\le \beta \right\}.
	\end{equation*}
	Hence it suffices to show that for any $\beta>0$,
	\begin{equation}\label{E:Hmeasure0}
		\cH^{ \sigma}_\beta( D_{\epsilon}\cap  B_r^\bM)=0.
	\end{equation}
	
	By applying Lemma \ref{L:prodcover} to $r>0$ and $\rho=\frac{1}{2}(\sigma-\alpha)$, we can find
	$T>0$ and $R>0 $ such that for any $\ell \in \NN$, the set
	\[
	D_{\delta'}(R,\ell T)\cap B_r^{\bM}
	\]
	can be covered by at most $e^{(\sigma+\alpha)\ell T/2}$ balls of radius $e^{-\ell T}$.
	Suppose that $\ell_1$ is large enough such that $2e^{-\ell_1 T}\le \beta$.
	By Lemma \ref{L:bddgap}, we see that
	\[
	D_\epsilon\cap B_r^\bM \subseteq \bigcup_{\ell \in \NN, \ell\ge \ell_1}
	D_{\delta'}(R, \ell T)\cap B_r^\bM.
	\]
	It follows that
	\begin{align*}
		\cH^{ \sigma}_\beta(D_{\epsilon}\cap B_r^\bM)
		&\le \sum_{\ell\ge \ell_1}  \cH^{ \sigma}_{\beta}( D _{\delta'}(R, \ell T)\cap B_r^\bM) \\
		&\le \sum_{\ell \ge \ell_1} e^{(\sigma+\alpha)\ell T/2} e^{- \sigma \ell T}\\
		&=\frac{e^{-\frac{1}{2}( \sigma-\alpha)\ell_1 T}}{1-e^{-\frac{1}{2}( \sigma-\alpha) T}}.
	\end{align*}
	By letting $\ell_1\to +\infty$, we get  \eqref{E:Hmeasure0}. This completes the proof.
\end{proof}

\section{The lower bound estimate}\label{s-lower}
The aim of this section is to establish the lower bound estimate Proposition \ref{P:mainlower}. In this section, we introduce a new type of essential singularity slightly stronger than $D^{e,s-1}$, and give the lower bound estimate of the set of points with such an essential singularity. 

\begin{definition}
	Let $X$ be given as in \eqref{E:notation}, and let $C\subseteq A^+$ be a cone. A point $\bx\in X$ is called \emph{fully essentially $C$-singular} if it is $C$-singular, but for any $1\leq i\leq s$ and any cone $C'$ in $A^+$ with $C'\setminus\{0\}\subseteq \mathrm{int}(A^+)$, $\pi_{\widehat{i}}(\bx)$ is not $\pi_{\widehat{i}}(C')$-singular.
\end{definition}

We denote by $\widetilde{D}^{e}(C,X)$ the set of fully essentially $C$-singular points in $X$. Since its definition involves more rays outside the cone, it is clear that $$\widetilde{D}^{e}(C,X)\subseteq D^{e,s-1}(C,X).$$
Similarly, we may define $$
\widetilde{D}^e(C, \bM):=\{\bTheta\in\bM:\bx_{\bTheta}\in \widetilde{D}^e(C,X)\}.
$$
We will prove the following result:


\begin{proposition}\label{p-lower-bound1}
Let $\ba=(a_1, \ldots, a_s)\in \RR^s_+$. Then for any $\epsilon\in (0, 1/64)$,
\[\dim \widetilde{D}^e(C_{V(\ba, \epsilon)}, \bM) \ge \sum_{i=1}^sm_in_i- \min_{1\le i\le s} \frac{m_in_i}{m_i+n_i}-8\epsilon \left(\sum_{i=1}^s\frac{m_in_i}{m_i+n_i}-\min_{1\leq i\leq s}\frac{m_in_i}{m_i+n_i}\right).\]
\end{proposition}

The main tool of the proof is the \emph{variational principle} in parametric geometry of numbers developed by Das, Fishman, Simmons and Urba\'nski in \cite{DFSU17,VarPrinc}. It allows us to construct a set of points with a given Diophantine property,
whose Hausdorff dimension is computable. Before heading to the proof of Proposition \ref{p-lower-bound1}, we first recall some basics of parametric geometry of numbers.

\subsection{Parametric geometry of numbers and variational principle}

Parametric geometry of numbers originates in a question of Schmidt \cite{SchmidtLuminy}. It was developed by Schmidt and Summerer \cite{SS09,SS13} and Roy \cite{Roy1}. Recently, Das, Fishman, Simmons and Urba\'nski \cite{DFSU17,VarPrinc} established a variational principle, which generalizes and quantifies an important theorem of Roy and has been a powerful tool for computing Hausdorff dimensions.

Let $(m,n)\in \NN^2$ and $\btheta \in M_{m\times n}(\RR)$.
The main purpose of parametric geometry of numbers is to study the trajectory $\{g_t^{(m,n)}x_\btheta: t\ge 0\}\subset Y_{m+n}$ through the \emph{successive minima function}
$$\bh = \bh_\btheta := (h_{\btheta,1}, \ldots , h_{\btheta,m+n}) : [0,\infty) \to \RR^{m+n}$$
where for $1\le k\le m+n$ and $t\ge0$,
\[ h_{\btheta,k}(t) = \log \lambda_k(g_t^{(m,n)}x_\btheta),\]
and $\lambda_k(\cdot)$ denotes the $k$-th successive minimum of a lattice in $\RR^{m+n}$.

It is easy to see that, up to a finite error, $\bh(\cdot)$ is piecewise linear with few possible slopes. Minkowski's first and second convex body theorems give further information. In a landmark paper \cite{Roy1}, Roy showed that when $m$ or $n$ is $1$, the successive minima functions are precisely approximated by \emph{Roy-systems}, which are relatively simple combinatorial objects. In \cite{DFSU17,VarPrinc}, Das, Fishman, Simmons and Urba\'nski extended this result to arbitrary $m$ and $n$ and quantified the result.

\begin{definition}[\cite{VarPrinc}]\label{d-template}
Let $(m,n)\in \NN^2$ and $I\subset [0,\infty)$ be an interval.
An \emph{$m \times n$ template on $I$} is a piecewise linear continuous map $\bL=(L_1, \ldots, L_{m+n}) : I \to \RR^{m+n}$ satisfying the following conditions:
\begin{enumerate}
\item $L_1 \le L_2 \le \cdots \le L_{m+n}$;
\item The derivative $L_j'(t)$, when well-defined at $t$, satisfies $-1/n\le L_j'(t) \le 1/m $;
\item For any $1\le j\le m+n$ and any subinterval $J\subset I$ such that $L_j < L_{j+1}$ on $J$ (with the convention $L_{m+n+1} = +\infty$), the restriction on $J$ of the function
$F_j := \sum_{0< k \le j }L_k$ is convex with slopes in the set
\[Z(j) := \left\{\frac{k_1}{m} - \frac{k_2}{n} :  0\le k_1 \le m, 0\le k_2\le n, k_1+k_2=j\right\}.\]
\end{enumerate}
\end{definition}

Generalizing Roy's theorem \cite{Roy1}, it is shown in \cite{VarPrinc} that for every $\btheta \in M_{m\times n}(\RR)$, there is an $m\times n$ template $\bL$ on $[0,\infty)$ such that $\bh_\btheta-\bL$ is bounded, and conversely, for every such template $\bL$, there exists some $\btheta \in M_{m\times n}(\RR)$ such that $\bh_\btheta - \bL$ is bounded.
The variational principle provides a quantitative version of the latter statement. It is expressed in terms of the \emph{lower average contraction rate} of a template, as described below.

For a template $\bL$ on $I\subset [0,\infty)$ and a subinterval $[T_1,T_2]\subset I$, one can define the \emph{average contraction rate} $\Delta(\bL, [T_1,T_2])$. As the definition is relatively long, we refer the reader to \cite[Definition 2.5]{VarPrinc}.
When $I=[0,\infty)$, we write $\Delta(\bL,T)=\Delta(\bL, [0,T])$, and define the \emph{lower average contraction rate} $\underline{\delta}(\bL)$ of $\bL$ as
\[\underline{\delta}(\bL) = \liminf_{T\to\infty} \Delta(\bL,T). \]

It is clear that the constant function $\bL=\bf{0}$ is a template, called the \emph{trivial} $m\times n$ template. Its lower average contraction rate is given below.

\begin{lemma}\label{r-trivial-dim}
The lower average contraction rate of the trivial $m\times n$ template on $[0,\infty)$ is $mn$.
\end{lemma}

\begin{proof}
This follows directly from \cite[Definition 2.5]{VarPrinc} (see also \cite[Section 28]{VarPrinc}).
\end{proof}

For an $m \times n$ template $\bL$ on $[0,\infty)$, set
\[\mathcal{M}(\bL) = \{ \btheta \in M_{m \times n}(\RR) : \bh_\btheta-\bL \textrm{ is bounded}\}.\]
More generally, given a collection $\mathcal{L}$ of $m \times n$ templates on $[0,\infty)$, denote
\[\mathcal{M}(\mathcal{L}) = \bigcup_{\bL\in\mathcal{L}} \mathcal{M}(\bL) . \]
The collection $\mathcal{L}$ is said to be \emph{closed under finite perturbations} if whenever $\bL$ and $\bL'$ are templates such that $\bL\in\mathcal{L}$ and $\bL-\bL'$ is bounded then $\bL'\in\mathcal{L}$.
The variational principle reads as follows.

\begin{theorem}[\cite{VarPrinc}]\label{t-dfsu2}
Let $\mathcal{L}$ be a Borel collection of $m \times n$ templates on $[0,\infty)$ that is closed under finite perturbations. Then
$$\dim \mathcal{M}(\mathcal{L})  = \sup_{\bL\in\mathcal{L}} \underline{\delta}(\bL).$$
\end{theorem}

\subsection{Reformulation of singular properties and the strategy of proof}

Let us fix an integer $s\ge 2$, a pair $(m_i, n_i)\in \NN^2$ for each $1\le i\le s$, and a weight vector $\ba=(a_1, \ldots, a_s)\in \RR_+^s$.
For simplicity, we write
\[b_i=\frac{m_in_i}{m_i+n_i}, \quad 1\le i\le s.\]
Without loss of generality, we may assume that
$$b_1=\min_{1\le i\le s} b_i.
$$
In view of Mahler's compactness criterion, we can reformulate the singular properties in terms of the successive minima function.

\begin{lemma}\label{p-smf}
Let $V\subset \RR^s_+$ be a nonempty closed convex set and $\bTheta=(\btheta_1, \ldots, \btheta_s)\in \bM$. Then we have
\begin{enumerate}
  \item $x_{\bTheta}$ is $C_V$-singular if and only if
\[\limsup_{t\to\infty} \max_{\bd=(d_1,\ldots, d_s)\in V}\min_{1\le i \le s} h_{\btheta_i,1}(d_i t) =-\infty.\]
  \item $x_{\bTheta}$ is fully essentially $C_V$-singular if and only if it is $C_V$-singular and for any cone $C'$ in $A^+$ with $C'\setminus\{0\}\subseteq \mathrm{int}(A^+)$,
\[\limsup_{t\to\infty} \min_{\substack{\bd=(d_1,\cdots,d_s)\in C'\\ \sum_{i=1}^sd_i=1}}\min_{\substack{1\le  i \le s \\ i\ne j}} h_{\btheta_i,1}(d_it)> - \infty \quad \text{ for all } 1\le j \le s.\]
\end{enumerate}
 \end{lemma}

\begin{proof}
These are direct consequences of the definitions of the singular properties and Mahler's compactness criterion.
\end{proof}

For the sake of convenience, when saying that $(\bL^1, \ldots , \bL^s)$ is an $s$-tuple of templates, we always assume that each $\bL^i$ is an $m_i\times n_i$ template on $[0,\infty)$, and write the $j$-th component of $\bL^i$ as $L^i_j$.

To prove Proposition \ref{p-lower-bound1}, we are going to construct $s$-tuples of templates $(\bL^1, \ldots , \bL^s)$ such that $\prod_{i=1}^s \cM(\bL^i)$ is contained in  $\widetilde{D}^e(C_V, \bM)$. Then the desired lower bound for the Hausdorff dimension of $\widetilde{D}^e(C_V, \bM)$ follows from Theorem \ref{t-dfsu2}. More precisely, we will construct tuples of templates satisfying the following lemma.

 \begin{lemma}\label{p-lower-key-1}
 For any $\epsilon\in (0, 1/64)$, there exists an $s$-tuple of templates $(\bL^1, \ldots , \bL^s)$ satisfying
\begin{eqnarray}
\underline{\delta}(\bL^1) &=& m_1n_1-b_1, \label{e-ess-dim-1}\\
\underline{\delta}(\bL^i) &\ge& m_in_i-8\epsilon b_i \quad \textrm{ for all } 2\le i \le s, \label{e-ess-dim-2}\\
\limsup_{t\to\infty}\max_{\bd\in V(\ba, \epsilon)}\min_{1\le i \le s}L^i_1(d_i t)&=& -\infty, \label{e-ess-div-1}\\
\limsup_{t\to\infty}\min_{\substack{\bd=(d_1,\cdots,d_s)\in C'\\ \sum_{i=1}^sd_i=1}}\min_{\substack{1\le i \le s  \\ i\ne j}}L^i_1(d_i t)& =& 0 \quad \textrm{ for all } 1\le j \le s \text{ and cones C'}. \label{e-ess-div-2}
\end{eqnarray}
Here the cones $C'$ are those cones in $A^+$ satisfying $C'\setminus\{0\}\subseteq \mathrm{int}(A^+)$.
\end{lemma}

We postpone the proof of Lemma \ref{p-lower-key-1} and first deduce Proposition \ref{p-lower-bound1} from it.

\begin{proof}[Proof of Proposition \ref{p-lower-bound1}]
Let $(\bL^1, \ldots , \bL^s)$ be an $s$-tuple of templates satisfying \eqref{e-ess-dim-1}--\eqref{e-ess-div-2}. By \eqref{e-ess-div-1}, \eqref{e-ess-div-2} and Lemma \ref{p-smf}, we have
\[\prod_{i=1}^s \cM(\bL^i) \subset \widetilde{D}^e(C_{V(\ba, \epsilon)}, \bM). \]
On the other hand, if we let $\mathcal{L}^i$ denote the collection of templates $\bL$ such that $\bL^i-\bL$ is bounded, then $\mathcal{L}^i$ is Borel and is closed under finite perturbations, and hence it follows from Theorem \ref{t-dfsu2} that
\[\dim \cM(\bL^i)=\dim \cM(\mathcal{L}^i)=\sup_{\bL\in\mathcal{L}^i} \underline{\delta}(\bL)\ge\underline{\delta}(\bL^i).\]
These properties, together with \eqref{e-ess-dim-1} and \eqref{e-ess-dim-2}, imply that
\begin{align*}
 \dim \widetilde{D}^e(C_{V(\ba, \epsilon)}, \bM) &\ge \dim \prod_{i=1}^s \cM(\bL^i) \ge \sum_{i=1}^s  \dim \cM(\bL^i) \\
 &\ge \sum_{i=1}^s  \underline{\delta}(\bL^i)\geq \sum_{i=1}^s  m_in_i - b_1-8\epsilon\cdot \sum_{i=2}^sb_i.
 \end{align*}
This completes the proof.
\end{proof}

\subsection{Standard templates}\label{construction}
The rest of this section is devoted to the proof of Lemma \ref{p-lower-key-1}. Our proof will be constructive.
In this subsection, we first recall the notion of \emph{standard template} defined by two points introduced in \cite{VarPrinc}, which will be the building blocks of our construction.

\begin{definition}[\cite{VarPrinc}]
For two points $(t', \eps'),(t'', \eps'')\in[0,\infty)^2$ with $t'<t''$, we write $\Delta t = t''-t'$, $\Delta \eps = \eps''-\eps'$. We say that the pair of points $((t', \eps'), (t'', \eps''))$ is \emph{admissible} if it satisfies the following conditions:
\begin{equation}\label{condST1}
-\frac{\Delta t}{m} \le \Delta \eps \le \frac{\Delta t}{n},
\end{equation}
\begin{equation}\label{condST2}
\Delta \eps \ge - \frac{n-1}{2n}\Delta t \; \textrm{ if } m=1, \textrm{ and } \Delta\eps \le \frac{m-1}{2m} \Delta t \; \textrm{ if } n=1,
\end{equation}
\begin{equation}\label{condST3}
(n-1)\Big(\frac{\Delta t}{n}  - \Delta \eps\Big) \ge (m+n) \eps'  \textrm{ \ or \ } (m-1) \Big(\frac{\Delta t}{m}  + \Delta \eps\Big ) \ge (m+n) \eps''.
\end{equation}
The \emph{standard template} $\bL((t', \eps'), (t'', \eps''))$ associated to an admissible pair $((t', \eps'), (t'', \eps''))$ is the $m\times n$ template $(L_1, \ldots, L_{m+n})$ on $[t',t'']$ defined in the following way.
\begin{itemize}
  \item Let $g_1, g_2 : [t', t''] \to \RR$ be piecewise linear functions such that
  $$g_1(t')=g_2(t') = -\eps', \quad g_1(t'')=g_2(t'')=-\eps'',$$
   and $g_i$ has two intervals of linearity: one on which $g'_i = 1/m$ and the other  on which $g'_i = - 1/n$ . For $i = 1$ the latter interval comes first while for $i = 2$ the former interval comes first. The existence of such functions $g_1$ and $g_2$ is guaranteed by \eqref{condST1}. Finally, let $g_3 = \cdots = g_{m+n}$ be functions on  $[t', t'']$ chosen so that $g_1(t) + \cdots + g_{m+n}(t) = 0$ for all $t\in [t', t'']$.
  \item Let $t \in [t', t'']$. If $g_2(t) \leq g_3(t)$, define $L_j(t)=g_j(t)$ for all $1\le j\le m+n$.
   Otherwise, define $L_1(t) = g_1(t)$, and define $L_2(t) = \cdots= L_{m+n}(t)$ so that $L_1(t)+\cdots+L_{m+n}(t) =0$.
\end{itemize}

Moreover, we say that a finite sequence of points $\{(t_l, \eps_l)\}_{1\le l \le k}$ is \emph{admissible} if for all $1\le l\le k-1$, the pair $((t_l, \eps_l), (t_{l+1}, \eps_{l+1}))$ is admissible. We define the \emph{standard template} associated to $\{(t_l, \eps_l)\}_{1\le l \le k}$ to be the template on the interval $[t_1, t_k]$ that equals $\bL((t_l, \eps_l), (t_{l+1}, \eps_{l+1}))$ on $[t_l, t_{l+1}]$.
\end{definition}

We will need the following statement.

\begin{lemma}\label{deltaST}
  Let $\bL$ be the standard template associated to an admissible pair of points $(t', \eps')$ and $(t'', \eps'')$. Then
  \begin{itemize}
    \item[(1)] $L_1(t)\le -\min \{\eps', \eps''\}$ for all $t\in [t', t'']$.
    \item[(2)] The average contraction rate on $[t',t'']$ is given by
\[ \Delta(\bL,[t',t'']) = mn-\frac{mn}{m+n} - O\left(\frac{\max(\eps', \eps'')}{t''-t'}\right).\]
  \end{itemize}
\end{lemma}

\begin{proof}
 (1) follows directly from the definition. For the proof of (2), see the paragraph below Definition 12.4 in \cite{VarPrinc}.
\end{proof}

The following simple observation will also be useful.

\begin{lemma}\label{l-adm}
Any pair of points $((t', \eps'), (t'', \eps''))$ satisfying
\begin{equation}\label{e-admiss}
  t''-t'\ge (m+n)^2 \max(\eps', \eps'')
\end{equation}
is admissible.
\end{lemma}
\begin{proof}
Given \eqref{e-admiss}, the conditions \eqref{condST1}, \eqref{condST2} and \eqref{condST3} are easily checked.
\end{proof}

\subsection{Construction of templates}
Now we begin to prove Lemma \ref{p-lower-key-1}. The task of this subsection is to construct the $s$-tuple of templates that we need. 

Set $T_0=1$ and $T_{k+1}=T_k+T_k^2$.   Write
$$l_k=\big\lceil T_k^{3/2}\,\big \rceil,\quad \text{ and }\quad  \gamma_k=l_k^{-1}T_k^2;$$
and for $0\le l\le l_k$, set $t_{k,l}=T_k + l\gamma_k$. Then we see that $t_{k,l_k}=T_{k+1}$.

Since $\gamma_k\geq T_k^{1/2}\to +\infty$ as $k\to +\infty$, there exists $k_0>0$ such that for any $k\ge k_0$,
\begin{equation}\label{e-valid}
  T_k \ge 32, \quad \text{ and } \quad a_i\gamma_k\ge (m_i+n_i)^2 \log \gamma_{k} \text{ for } i=1,2.
\end{equation}

For any $k\ge k_0$, set
\[r_k=\max \left\{1\le l\le l_k: t_{k,l} \le T_{k}^{3/2} \right\},\quad
\text{and}\quad T_k'=t_{k, r_k}.\] 
Since $t_{k,1}<T_k^{3/2}<t_{k,l_k}$, we see that $r_k\in [2,l_k-1]$ is well-defined. It follows that $T_k'\leq T_k^{3/2}<T_k'+\gamma_k$. We also set
\[p_k=\max \left\{1\le l\le l_k: t_{k,l} \le (1-4\epsilon) t_{k, l_k-1} \right\},\]
  and
  \[ q_k=\min \left\{1\le l\le l_{k}: t_{k, l} \ge (1+4\epsilon)t_{k, 1} \right\}.\]
 It is clear that $p_k\in [l_k/2,l_k-2]$ and $q_k\in [2,l_k/2-1]$. Moreover, we see that \begin{equation}\label{E:defpq}
 	t_{k,p_k} \le (1-4\epsilon) t_{k, l_k-1}<t_{k,p_k}+\gamma_k,\quad \text{and}\quad t_{k, q_k} \ge (1+4\epsilon)t_{k, 1}>t_{k, q_k}-\gamma_k.
 \end{equation}

A direct computation gives the following lemma which will be used later.
\begin{lemma}\label{l-comput}
Let $k\ge k_0$. Then we have
\begin{itemize}
  \item[(1)] $t_{k, q_k+1}<T_k'< t_{k, p_k-1}$; moreover, for any $C\geq 1$, when $k$ is large enough,  $$Ct_{k, q_k+1}<T_k'< C^{-1}t_{k, p_k-1};$$
  \item[(2)] $(1+\epsilon)t_{k, p_k} \le (1-\epsilon)t_{k, l_k-1}$;
  \item[(3)] $(1+\epsilon)t_{k, 1} \le (1-\epsilon)t_{k, q_k}$.
\end{itemize}
\end{lemma}
\begin{proof}
The conclusion (2) and (3) follow directly from the definitions of $p_k$ and $q_k$. For (1), let $C\geq 1$. Note that on one hand,
\begin{align*}
  T_k'-Ct_{k, q_k+1} &>(T_{k}^{3/2}-\gamma_k)-C((1+4\epsilon)(T_k+\gamma_k)+2\gamma_k)\\
  &> T_{k}^{3/2}-C(2T_k+5\sqrt{T_k})>0.
\end{align*}
On the other hand,
\begin{align*}
  t_{k, p_k-1}-CT_k' &> ((1-4\epsilon)(T_{k+1}-\gamma_k)-2\gamma_k)-CT_{k}^{3/2}\\
  &> \frac{1}{2}T_k^2-3\sqrt{T_k}-CT_{k}^{3/2}>0.
\end{align*}
This completes the proof of (1).
\end{proof}

  We define $\bL^1$ as follows (see Figure \ref{fig1}):
\begin{itemize}
  \item On the interval $[0, a_1T_{k_0}]$, set  $\bL^1$ to be the trivial template.
  \item On the interval $[a_1T_k, a_1T_{k+1}]$ for $k\ge k_0$, set $\bL^1$ to be the standard template associated to the sequence of points
      \[(a_1 T_k,0), (a_1 t_{k,1},\log \gamma_k), \ldots, (a_1 t_{k,l_k-1}, \log \gamma_k), (a_1 T_{k+1},0).\]
\end{itemize}
According to Lemma \ref{l-adm} and \eqref{e-valid}, the sequence of points is admissible, hence the construction is valid.

For $2\le i\le s$, we define $\bL^i$ as follows (see Figure \ref{fig1}):
\begin{itemize}
  \item On the interval $[0, a_iT_{k_0}']$, set  $\bL^i$ to be the trivial template.
  \item When $k\ge k_0$ and $k\ne i-1 \mod{(s-1)}$, set $\bL^i$ to be the trivial template on $[a_iT_k', a_i T_{k+1}']$.
  \item When $k\ge k_0$ and $k= i-1 \mod{(s-1)}$, set $\bL^i$ to be the following template on $[a_iT_k', a_i T_{k+1}']$:
   \begin{itemize}
         \item[(i)] On the subinterval $[a_iT_k', a_it_{k,p_k-1}]$,  set $\bL^i$ to be the trivial template.
         \item[(ii)] On the subinterval $[a_it_{k,p_k-1}, a_iT_{k+1}]$, set $\bL^i$ to be the standard template associated to the sequence of points
       \[(a_i t_{k,p_k-1},0), (a_i t_{k,p_k},\log \gamma_k),\ldots,  (a_i t_{k,l_k-1}, \log \gamma_k), (a_i T_{k+1},\log \gamma_k).\]
         \item[(iii)] On the subinterval $[a_iT_{k+1}, a_it_{k+1,q_{k+1}+1}]$, set $\bL^i$ to be the standard template associated to the sequence of points
       \[(a_i T_{k+1},\log \gamma_k), (a_i t_{k+1,1},\log \gamma_k),\ldots,  (a_i t_{k+1,q_{k+1}}, \log \gamma_k), (a_i t_{k+1,q_{k+1}+1},0).\]
       \item[(iv)] On the subinterval $[a_it_{k+1,q_{k+1}+1}, a_iT_{k+1}']$,  set $\bL^i$ to be the trivial template.
       \end{itemize}

      \end{itemize}
  According to Lemma \ref{l-adm} and \eqref{e-valid}, the sequences of points given above are admissible, hence the construction is valid.
  \begin{remark}
  Here, the auxiliary timepoints $T_k'$ are introduced for the convenience of treating the cases $s=2$ and $s>2$ in a unified manner.
  \end{remark}

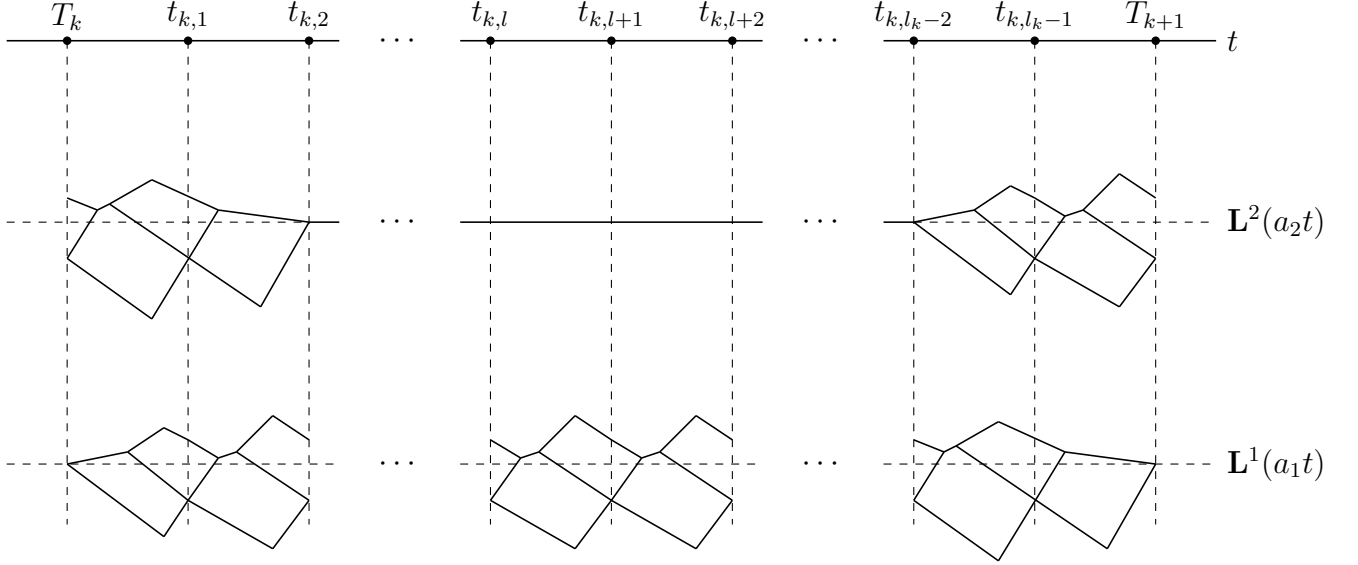
\begin{figure}[h!]
 \begin{center}
 \begin{tikzpicture}[scale=0.8]
 \draw[black, semithick] (14.5,0)--(20,0) node [right,black] { $t$ } ;
 \draw[black, semithick] (7.5,0)--(12.5,0) node [right,black] { } ;
 \draw[black, semithick] (0,0)--(5.5,0) node [right,black] {  } ;

 \draw (6.5,0) node{$\cdots$};
  \draw (13.5,0) node{$\cdots$};

  \fill (1,0) circle[radius=2pt] node [above,black] { $T_k$ } ;
    \draw[black, dashed] (1,0)--(1,-8) ;
  \fill (19,0) circle[radius=2pt] node [above,black] { $T_{k+1}$ } ;
   \draw[black, dashed] (19,0)--(19,-8) ;
  \fill (3,0) circle[radius=2pt] node [above,black] { $t_{k,1}$ } ;
    \draw[black, dashed] (3,0)--(3,-8) ;
  \fill (5,0) circle[radius=2pt] node [above,black] { $t_{k,2}$ } ;
    \draw[black, dashed] (5,0)--(5,-8) ;
  \fill (17,0) circle[radius=2pt] node [above,black] { $t_{k,l_k-1}$ } ;
      \draw[black, dashed] (17,0)--(17,-8) ;
  \fill (15,0) circle[radius=2pt] node [above,black] { $t_{k,l_k-2}$ } ;
    \draw[black, dashed] (15,0)--(15,-8) ;
   \fill (8,0) circle[radius=2pt] node [above,black] { $t_{k,l}$ } ;
     \draw[black, dashed] (8,0)--(8,-8) ;
    \fill (12,0) circle[radius=2pt] node [above,black] { $t_{k,l+2}$ } ;
      \draw[black, dashed] (12,0)--(12,-8) ;
         \fill (10,0) circle[radius=2pt] node [above,black] { $t_{k,l+1}$ } ;
      \draw[black, dashed] (10,0)--(10,-8) ;


     \draw[black, semithick] (5,-3)--(5.5,-3) ;
     \draw[black, semithick] (7.5,-3)--(12.5,-3) ;
     \draw[black, semithick] (14.5,-3)--(15,-3) ;

     \draw (6.5,-3) node{$\cdots$};
     \draw (13.5,-3) node{$\cdots$};

      \draw[black, dashed] (15,-3)--(20,-3) node [right,black] { $\bL^2(a_2 t)$ } ;
       \draw[black, dashed] (0,-3)--(5,-3) ;


       \draw[black, semithick] (1,-2.6)--(1.5,-2.8);
      \draw[black, semithick] (1,-3.6)--(1.5,-2.8);
      \draw[black, semithick] (1,-3.6)--(2.4,-4.6);
      \draw[black, semithick] (2.4,-4.6)--(3.5,-2.8);
      \draw[black, semithick] (1.5,-2.8)--(1.7,-2.7);
      \draw[black, semithick] (1.7,-2.7)--(2.4,-2.3);
      \draw[black, semithick] (1.7,-2.7)--(4.2,-4.4);
      \draw[black, semithick] (2.4,-2.3)--(3.5,-2.8);
      \draw[black, semithick] (3.5,-2.8)--(5,-3);
      \draw[black, semithick] (4.2,-4.4)--(5,-3);


      \draw[black, semithick] (15,-3)--(16,-2.8);
      \draw[black, semithick] (15,-3)--(16.6,-4.2);
      \draw[black, semithick] (16,-2.8)--(16.6,-2.4);
      \draw[black, semithick] (16,-2.8)--(17,-3.6);
      \draw[black, semithick] (16.6,-4.2)--(17,-3.6);
      \draw[black, semithick] (16.6,-2.4)--(17,-2.6);

      \draw[black, semithick] (17,-2.6)--(17.5,-2.9);
      \draw[black, semithick] (17,-3.6)--(17.5,-2.9);
      \draw[black, semithick] (17,-3.6)--(18.4,-4.4);
      \draw[black, semithick] (18.4,-4.4)--(19,-3.6);
      \draw[black, semithick] (17.5,-2.9)--(17.8,-2.8);
      \draw[black, semithick] (17.8,-2.8)--(19,-3.6);
      \draw[black, semithick] (18.4,-2.2)--(19,-2.6);
      \draw[black, semithick] (17.8,-2.8)--(18.4,-2.2);

      \draw[black, dashed] (0,-7)--(5.5,-7) ;
      \draw[black, dashed] (7.5,-7)--(12.5,-7) ;
      \draw[black, dashed] (14.5,-7)--(20,-7)node [right,black] { $\bL^1(a_1 t)$ }  ;

      \draw (6.5,-7) node{$\cdots$};
      \draw (13.5,-7) node{$\cdots$};


       \draw[black, semithick] (1,-7)--(2,-6.8);
      \draw[black, semithick] (1,-7)--(2.6,-8.2);
      \draw[black, semithick] (2,-6.8)--(2.6,-6.4);
      \draw[black, semithick] (2,-6.8)--(3,-7.6);
      \draw[black, semithick] (2.6,-8.2)--(3,-7.6);
      \draw[black, semithick] (2.6,-6.4)--(3,-6.6);

      \draw[black, semithick] (3,-6.6)--(3.5,-6.9);
      \draw[black, semithick] (3,-7.6)--(3.5,-6.9);
      \draw[black, semithick] (3,-7.6)--(4.4,-8.4);
      \draw[black, semithick] (4.4,-8.4)--(5,-7.6);
      \draw[black, semithick] (3.5,-6.9)--(3.8,-6.8);
      \draw[black, semithick] (3.8,-6.8)--(5,-7.6);
      \draw[black, semithick] (4.4,-6.2)--(5,-6.6);
      \draw[black, semithick] (3.8,-6.8)--(4.4,-6.2);

      \draw[black, semithick] (8,-6.6)--(8.5,-6.9);
      \draw[black, semithick] (8,-7.6)--(8.5,-6.9);
      \draw[black, semithick] (8,-7.6)--(9.4,-8.4);
      \draw[black, semithick] (9.4,-8.4)--(10,-7.6);
      \draw[black, semithick] (8.5,-6.9)--(8.8,-6.8);
      \draw[black, semithick] (8.8,-6.8)--(10,-7.6);
      \draw[black, semithick] (9.4,-6.2)--(10,-6.6);
      \draw[black, semithick] (8.8,-6.8)--(9.4,-6.2);

      \draw[black, semithick] (10,-6.6)--(10.5,-6.9);
      \draw[black, semithick] (10,-7.6)--(10.5,-6.9);
      \draw[black, semithick] (10,-7.6)--(11.4,-8.4);
      \draw[black, semithick] (11.4,-8.4)--(12,-7.6);
      \draw[black, semithick] (10.5,-6.9)--(10.8,-6.8);
      \draw[black, semithick] (10.8,-6.8)--(12,-7.6);
      \draw[black, semithick] (11.4,-6.2)--(12,-6.6);
      \draw[black, semithick] (10.8,-6.8)--(11.4,-6.2);

      \draw[black, semithick] (15,-6.6)--(15.5,-6.8);
      \draw[black, semithick] (15,-7.6)--(15.5,-6.8);
      \draw[black, semithick] (15,-7.6)--(16.4,-8.6);
      \draw[black, semithick] (16.4,-8.6)--(17.5,-6.8);
      \draw[black, semithick] (15.5,-6.8)--(15.7,-6.7);
      \draw[black, semithick] (15.7,-6.7)--(16.4,-6.3);
      \draw[black, semithick] (15.7,-6.7)--(18.2,-8.4);
      \draw[black, semithick] (16.4,-6.3)--(17.5,-6.8);
      \draw[black, semithick] (17.5,-6.8)--(19,-7);
      \draw[black, semithick] (18.2,-8.4)--(19,-7);

  \end{tikzpicture}
 \end{center}
 \caption{$\bL^1(a_1 t)$ and $\bL^2(a_2 t)$ on the interval $[T_k,T_{k+1}]$, $k\ge k_0$. }\label{fig1}
 \end{figure}

\subsection{Verification of conditions in Lemma \ref{p-lower-key-1}}\label{s-4}
The task of this subsection is to verify that the $s$-tuple of templates constructed in the last subsection will satisfy the conditions of Lemma \ref{p-lower-key-1}. The verification will be divided into three paragraphs.

\subsubsection{Verification of \eqref{e-ess-dim-1} and \eqref{e-ess-dim-2}}
In this paragraph, we compute the lower average contraction rates of the templates $\bL^1,\bL^2,\cdots,\bL^s$.

Let $1\leq i\leq s$, $k\ge k_0$ and $0\le l\le l_k-1$. Then for any $T\in [a_it_{k,l}, a_it_{k, l+1}]$, 
\[\Delta(\bL^i, T) - \Delta(\bL^i, a_it_{k,l}) = O \left(\frac{T- a_it_{k,l}}{a_it_{k,l}}\right) = O \left( \frac{1}{T_k^{1/2}}\right), \]
which goes to zero as $k$ tends to infinity.
Hence it suffices to compute $\Delta(\bL^i,T)$ at the timepoints $a_it_{k, l}$.  

By definition, we see that for $0\leq l\leq l_{k+1}-1$,
\begin{align*}
 \Delta\left(\bL^i, a_it_{k+1,l}\right)&=\frac{T_k}{t_{k+1, l}}\Delta\left(\bL^i, a_iT_k\right)
  +\frac{t_{k+1, l}-T_{k}}{t_{k+1, l}}\Delta\left(\bL^i, [a_iT_{k}, a_it_{k+1, l}]\right).
 \end{align*}
Since $T_k/t_{k+1, l}$ goes to zero as $k$ goes to infinity, it suffices to study $\Delta\left(\bL^i, [a_iT_{k}, a_it_{k+1, l}]\right)$.
For simplicity, we write
\[\Delta(i, k', l')=\Delta\left(\bL^i, [a_it_{k', l'}, a_it_{k', l'+1}]\right).\]
Note that the interval $[a_iT_{k}, a_it_{k+1, l}]$ can be decomposed as $$
\bigcup_{l'=0}^{l_k-1}[a_it_{k,l'},a_it_{k,l'+1}]\cup \bigcup_{l''=0}^{l-1}[a_it_{k+1,l''},a_it_{k+1,l''+1}].
$$
It follows that the quantity $\Delta\left(\bL^i, [a_iT_{k}, a_it_{k+1, l}]\right)$ can be written as
\begin{equation}\label{e-sum-delta}
  \frac{1}{t_{k+1, l}-T_k}\left(\gamma_k\sum_{0\le l'\le l_k-1}\Delta(i, k, l') + \gamma_{k+1}\sum_{0\le l''\le l-1}\Delta(i, k+1, l'')\right).
\end{equation}
There are two cases to consider.

\begin{itemize}
  \item[(i)] Case when $i=1$.
\end{itemize}
 
   According to Lemma \ref{deltaST} (2) and our construction, for any $k'\ge k_0$ and any $0\le l'\le l_{k'}-1$,  we have
  \begin{equation*}\label{e-de111}
    \Delta(1, k', l')= m_1n_1-b_1-O\left(\frac{\log \gamma_{k'}}{\gamma_{k'}}\right),
  \end{equation*}
   which converges to $m_1n_1-b_1$ as $k'$ goes to infinity. In view of \eqref{e-sum-delta}, the quantity $\Delta\left(\bL^i, [a_iT_{k}, a_it_{k+1, l}]\right)$ equals the weighted average of $\Delta(1, k, l')$ and $\Delta(1, k+1, l'')$. Hence taking $k\to +\infty$ gives $$
   \Delta\left(\bL^i, [a_iT_{k}, a_it_{k+1, l}]\right)\to m_1n_1-b_1.
   $$
   This verifies \eqref{e-ess-dim-1}.

\begin{itemize}
   \item[(ii)]  Cases when $2\le i\le s$. 
\end{itemize}
   
    According to Lemma \ref{r-trivial-dim}, Lemma \ref{deltaST} (2) and our construction, for any $k'\ge k_0$, when $0\le l'\le q_{k'}$ or $p_{k'}-1\le l\le l_{k'}-1$,  we have
      \begin{equation*}
    \Delta\left(i, k', l'\right)\in\left\{ m_in_i ,\; m_in_i-b_i-O\left(\frac{\log \gamma_{k'}}{\gamma_{k'}}\right)\right\};
    \end{equation*} when $q_{k'}+1 \le l'\le p_{k'}-2$, we have
  \begin{equation*}
    \Delta\left(i, k', l'\right)= m_in_i.
  \end{equation*}
  
  To collect these data, we write
  \begin{equation*}
  	\delta(i,k',l'):=\begin{cases}
  		m_in_i-b_i, & \mbox{if } 0\le l'\le q_{k'} \text{ or } p_{k'}-1\le l'\le l_{k'}-1 \\
  		m_in_i, & \mbox{if } q_{k'}+1\le l'\le p_{k'}-2
  	\end{cases},
  \end{equation*}
  and consider the weighted sum  $$
  S(k+1,l):=\frac{1}{t_{k+1, l}-T_k}\left(\gamma_k\sum_{0\le l'\le l_k-1}\delta(i, k, l') + \gamma_{k+1}\sum_{0\le l''\le l-1}\delta(i, k+1, l'')\right).
  $$
  In view of \eqref{e-sum-delta}, we see that $$
  \lim_{k\to +\infty}\Delta\left(\bL^i, [a_iT_{k}, a_it_{k+1, l}]\right)\geq \lim_{k\to \infty}S(k+1,l).$$
  Therefore, to verify \eqref{e-ess-dim-2}, it suffices to show that
  \begin{equation}\label{e-234}
    \lim_{k\to \infty}\min \{S(k+1, l): 0\le l\le l_{k+1}-1\}\ge m_in_i-8\epsilon b_i.
  \end{equation}

  Indeed, a direct calculation shows that the function $$
  [0,l_{k+1}-1]\to \RR_+,\quad l\mapsto S(k+1, l)$$ 
  decreases when $0\le l\le q_{k+1}$ or $p_{k+1}-1\le l\le l_{k+1}-1$, and increases when $q_{k+1}+1\le l\le p_{k+1}-2$. Thus its minimum is attained at either $q_{k+1}$ or $l_{k+1}-1$. At these two timepoints, we have
  \begin{align*}
    S(k+1, q_{k+1}) &\,\;=m_in_i-\left( \frac{\gamma_k(l_k-p_k+q_k+2)}{t_{k+1, q_{k+1}}-T_k}+ \frac{\gamma_{k+1}q_{k+1}}{t_{k+1, q_{k+1}}-T_k}\right)b_i\\
    &\overset{\eqref{E:defpq}}{\geq} m_in_i-\left( \frac{4\epsilon T_{k+1}+4\epsilon T_k}{(1+4\epsilon)T_{k+1}}+\frac{4\epsilon T_{k+1}}{(1+4\epsilon )T_{k+1}} \right) b_i+o(1)\\
    &\;\;\ge m_in_i-8\epsilon b_i+o(1),
      \end{align*}
  and
  \begin{align*}
    S(k+1, l_{k+1}-1) &\,\;=m_in_i-\left( \frac{\gamma_{k+1}(l_{k+1}-p_{k+1}+q_{k+1})}{t_{k+1,l_{k+1}-1}-T_k}\right)b_i+o(1)\\
    &\overset{\eqref{E:defpq}}{\geq} m_in_i-\left( \frac{4\epsilon T_{k+2}+4\epsilon T_{k+1}}{T_{k+2}}\right) b_i+o(1)\\
    &\;\;\ge m_in_i-4\epsilon b_i+o(1).
  \end{align*}
  This completes the proof of \eqref{e-234}, and hence verifies \eqref{e-ess-dim-2}.

\subsubsection{Verification of \eqref{e-ess-div-1}} In this paragraph, we focus on the function $\min_{1\leq i\leq s}L^i_1(\cdot)$. Let $k\geq k_0$. In view of Lemma \ref{deltaST} (1) and the construction of $\bL^i$, for $i=1$ we have
\[L^1_1(a_1t)  \le -\log \gamma_k \text{ on the interval } [t_{k,1}, t_{k, l_k-1}];\]
for $2\leq i\leq s$ and $k=i-1 \mod{(s-1)}$, we have
 \[L^i_1(a_it)  \le -\log \gamma_k \text{ on the interval } [t_{k,p_k}, t_{k+1, q_{k+1}}].\]
It follows that for any $\bd\in V(\ba,\epsilon)$, for $i=1$ we have
\begin{equation}\label{e-ine123}
  L^1_1(d_1t)\le -\log \gamma_k \text{ on } [(1+\epsilon)t_{k,1}, (1-\epsilon)t_{k, l_k-1}];
\end{equation}
for $2\leq i\leq s$ and $k=i-1 \mod{(s-1)}$, we have
 \begin{equation}\label{e-ine456}
   L^i_1(d_it)  \le -\log \gamma_k \text{ on  } [(1+\epsilon)t_{k,p_k}, (1-\epsilon)t_{k+1, q_{k+1}}].
 \end{equation}
 
 Note that Lemma \ref{l-comput} (2)(3) imply that the interval $[(1+\epsilon)t_{k,1}, (1-\epsilon)t_{k+1, l_{k+1}-1}]$ can be written as the union of three intervals $$
 [(1+\epsilon)t_{k,1},(1-\epsilon)t_{k,l_k-1}]
 \cup [(1+\epsilon)t_{k,p_k},(1-\epsilon)t_{k+1,q_{k+1}}]
 \cup [(1+\epsilon)t_{k+1,1},(1-\epsilon)t_{k+1,l_{k+1}-1}].
 $$
 For any $\bd\in V(\ba,\epsilon)$, by applying \eqref{e-ine123}, \eqref{e-ine456} for $k, k+1$ to above three intervals, we get 
 \[\min_{1\le i\le s} L^i_1(d_i t) \le -\log \gamma_k \text{ on } [(1+\epsilon)t_{k,1}, (1-\epsilon)t_{k+1, l_{k+1}-1}].\]
 According to Lemma \ref{l-comput}, this interval contains $[T_k', T_{k+1}']$ as a subinterval.
 Hence,
 \begin{align*}
   \limsup_{t\to +\infty}\max_{\bd\in V(\ba,\epsilon)}\min_{1\le i \le s}L^i_1(d_i t) \le \limsup_{k\to +\infty}(-\log\gamma_k) =-\infty.
 \end{align*}
 This completes the proof of \eqref{e-ess-div-1}.

\subsubsection{Verification of \eqref{e-ess-div-2}} In this paragraph, for each $1\leq j\leq s$, we focus on the function $\min_{\substack{1\leq i\leq s\\ i\neq j}}L^i_1(\cdot)$. According to the definition of templates, the upper bound is trivial. Hence it suffices to prove the lower bound. There are two cases to consider.

\begin{enumerate}
	\item[(i)] Case when $j=1$.
\end{enumerate}
	
	 Let $k\geq k_0$. By construction, for $2\le i\le s$, we have
\[ L^i_1(a_i t)=0 \text{ on  } [t_{k,q_k+1}, t_{k, p_{k}-1}].  \]
For any cone $C'$ in $A^+$ with $C'\setminus\{0\}\subseteq \mathrm{int}(A^+)$, set
\[\kappa_{C'}=\max_{\substack{\bd=(d_1,\cdots,d_s)\in C'\\\sum_{i=1}^sd_i=1}}\max_{1\le i\le s}\max\left\{\frac{a_i}{d_i}, \frac{d_i}{a_i}\right\}>0.\]
It follows that for any $2\le i\le s$ and any $\bd\in C'$ with $\sum_{i=1}^sd_i=1$, we have
\[ L^i_1(d_i t)=0 \text{ on  } [\kappa_{C'} t_{k,q_k+1}, \kappa_{C'}^{-1} t_{k, p_{k}-1}].  \]
According to Lemma \ref{l-comput} (1), for all sufficiently large $k$ (depending on $\kappa_{C'}$), we have
\[T_k'\in [\kappa_{C'} t_{k,q_k+1}, \kappa_{C'}^{-1} t_{k, p_{k}-1}],\]
which in particular gives $L^i_1(d_i T_k')=0$.
 Therefore, we conclude that
\begin{align*}
  \limsup_{t\to\infty}\min_{\substack{\bd\in C'\\\sum_{i=1}^sd_i=1}}\min_{\substack{2\le i \le s }}L^i_1(d_i t)
  \ge \limsup_{\substack{k\to\infty }}\min_{\substack{\bd\in C'\\\sum_{i=1}^sd_i=1}}\min_{\substack{2\le i \le s  }}L^i_1(d_i T_k')
  =0.
  \end{align*}
  
 \begin{enumerate}
\item[(ii)] Cases when $2\le j\le s$.
\end{enumerate}

Let $k\geq k_0$. By construction, for $i=1$ we always have \[L^1_1(a_1T_k)=0.\]
Moreover, for any $2\leq i\leq s$ with $i\neq j$, and any $k= j-1 \mod{(s-1)}$,  we have 
\[ L^i_1(a_i t)=0 \text{ on } [T_k', T_{k+1}'].\]
Consequently, for any cone $C'$ as above, any $\bd\in C'$ with $\sum_{i=1}^sd_i=1$, any $2\leq i\leq s$ with $i\neq j$, and any $k= j-1 \mod{(s-1)}$,  we have 
\[ L^i_1(d_i t)=0 \text{ on  } [\kappa_{C'} T_k', \kappa_{C'}^{-1} T_{k+1}'].  \]
According to Lemma \ref{l-comput} (1), for all sufficiently large $k$ (depending on $\kappa_{C'}$), we have
\[\frac{a_1}{d_1}T_{k+1}\in [\kappa_{C'} T_k', \kappa_{C'}^{-1} T_{k+1}'],\]
which in particular gives $L^i_1(d_i\cdot \frac{a_1}{d_1} T_{k+1})=0$. 

Therefore, we conclude that
\begin{align*}
  \limsup_{t\to\infty}\min_{\substack{\bd\in C'\\\sum_{i=1}^sd_i=1}}\min_{\substack{1\le i \le s  \\ i\ne j}}L^i_1(d_i t) 
  \ge \limsup_{\substack{k\to\infty \\ k=j-1 \mod s-1}}\min_{\substack{\bd\in C'\\\sum_{i=1}^sd_i=1}}\min_{\substack{1\le i \le s  \\ i\ne j}}L^i_1\left(d_i\cdot \frac{a_1}{d_1} T_{k+1}\right)=0.
  \end{align*}
  This completes the proof of \eqref{e-ess-div-2}.



\end{document}